\documentclass[11pt]{amsart}
\usepackage{amsfonts,amssymb,graphicx, enumerate,youngtab}
\usepackage{mathtools, pdflscape, rotating}

\usepackage[american]{babel}

\usepackage{fourier} 
\usepackage{array}
\usepackage{makecell}

\usepackage{longtable}

\input prepictex
\input pictex
\input postpictex

\numberwithin{equation}{section}

\newtheorem{theorem}{Theorem}[section]
\newtheorem{lemma}[theorem]{Lemma}

\newtheorem{corollary}[theorem]{Corollary}
\newtheorem{hypothesis}[theorem]{Hypothesis}

\theoremstyle{definition}

\theoremstyle{remark}

\newcommand{\CC}{\mathbf{C}}

\newcommand{\QQ}{\mathbb{Q}}
\newcommand{\ZZ}{\mathbf{Z}}

\newcommand{\OO}{\mathcal{O}}

\newcommand{\la}{\langle}
\newcommand{\ra}{\rangle}

\newcommand{\q}{\mathrm{q}}
\newcommand{\cc}{\mathrm{c}}

\newcommand{\eu}{\mathrm{eu}}
\newcommand{\triv}{\mathbf{1}}

\begin{document}

\title[The support of the spherical representation]{$W$-exponentials, Schur elements, and the support of the spherical representation of the rational Cherednik algebra}

\author{Stephen Griffeth}
\address{Stephen Griffeth \\ Instituto de Matem\'atica y F\'isica \\ Universidad de Talca \\ Campus Norte, Camino Lircay S/N \\ Talca, Chile
}
\email{sgriffeth@inst-mat.utalca.cl}

\author{Daniel Juteau}
\address{Daniel Juteau \\ CNRS, 
Institut de Math\'ematiques de Jussieu - Paris Rive Gauche \\
UMR7586, 
Univ. Paris Diderot, Sorbonne Paris Cit\'e,
Sorbonne Universit\'es \\
UPMC Univ. Paris 06, 
F-75013 Paris, France }
\email{daniel.juteau@imj-prg.fr}

\begin{abstract}
Given a complex reflection group $W$ we compute the support of the spherical irreducible module $L_c(\triv)$ of the rational Cherednik algebra $H_c(W)$ in terms of the simultaneous eigenfunction of the Dunkl operators and Schur elements for finite Hecke algebras.
\end{abstract}

\thanks{We thank Jean Michel for his work developing CHEVIE and his expert advice, which was indispensable for the explicit calculations at the end of the paper. We thank Maria Chlouveraki for her work implementing her calculations of Schur elements, and Thomas Gerber and Emily Norton for interesting discussions related to their recent preprint. We are especially grateful to Ivan Losev for explanations around biadjointness. We acknowledge the financial support of ANR grants VARGEN (ANR-13-BS01-0001-01) and GeRepMod (ANR-16-CE40-0010-01), Fondecyt Proyecto Regular 1151275 and MathAmSud grant RepHomol, which funded a visit by the second author to Talca in December 2016. Parts of this paper were written at the ICTP in Trieste, which we thank for an excellent working environment.}

\maketitle

\section{Introduction} 

In this paper we compute the support of the (unique) simple quotient $L_c(\triv)$ of the polynomial representation $\CC[V]=\Delta_c(\triv)$ of the rational Cherednik algebra of a complex reflection group $W$, and in particular determine when it is finite dimensional, in terms of the joint eigenfunction for the Dunkl operators (the ``$W$-exponential function'') and Schur elements for finite Hecke algebras, which arise in this context because they compute the endomorphism of the identity functor produced by an induction-restriction biadjunction. 

We refer to the body of the paper for detailed definitions. Given a complex reflection group $W$ acting in a vector space $V$, one obtains a stratification of $V$ whose strata are the equivalence classes for the equivalence relation $p \equiv q$ if the stabilizer groups are equal, $W_p=W_q$. Given a stratum $S$ we will write $W_S=W_p$ for any $p \in S$. The support of the irreducible head of the polynomial representation is the closure of a single $W$-orbit of strata. The parameter $c$ runs over a space $\mathcal{C}$ of parameters with coordinate functions that we label $c_{H,\chi}$, indexed by $W$-orbits of pairs consisting of a reflecting hyperplane $H$ for $W$ and a non-trivial linear character $\chi$ of the cyclic reflection subgroup $W_H$. A \emph{positive hyperplane} in $\mathcal{C}$ is a hyperplane $C \subseteq \mathcal{C}$ of the form
$$C=\left\{\sum a_{H,\chi} c_{H,\chi}=a \right\} \quad \hbox{with real $a_{H,\chi} \geq 0$ and $a >0$.}$$ The $W$-exponential function $e(x,\lambda,c)$ is a joint eigenfunction for the Dunkl operators with eigenvalue $\lambda$, normalized by $e(0,\lambda,c)=1$. For $\lambda$ fixed, there is an entire holomorphic renormalization function $F_S(c)$ with zero set depending only on the stratum $S$ of $\lambda$ with the properties that $F_S(c) e(x,\lambda,c)$ is entire as a function of $c$ and, for $c$ fixed, non-zero as a function of $x$. The zero set of $F_S(c)$ is a certain set of positive hyperplanes. We have: 

\begin{theorem} \label{ren support theorem}
A stratum $S$ is in the support of $L_c(\triv)$ if and only if $F_S(c) \neq 0$.
\end{theorem}

This theorem reduces the problem of calculating the support of $L_c(\triv)$ to that of calculating the zeros of $F_S(c)$ for all strata $S$. In order to do this we use the finite Hecke algebra $\mathcal{H}$ of $W$, and given a stratum $S$, the finite Hecke algebra $\mathcal{H}_S$ of the parabolic subgroup $W_S$ of $W$. Combining recent work of Bezrukavnikov-Etingof, Losev, Marin-Pfeiffer, and Shan shows that induction and restriction for the pair $\mathcal{H}$ and $\mathcal{H}_S$ are biadjoint, and the biadjunction produces a relative principal Schur element which we will write as  $s(q_c)=|W:W_S|_{q_c}$, whose zeros detect precisely when the trivial representation is not relatively $\mathcal{H}_S$-projective.

Our second main theorem is then an analog of the fact that the trivial representation of a finite group in characteristic $p$ is projective relative to a subgroup if and only if $p$ does not divide the index.

\begin{theorem} \label{support theorem1}
A stratum $S$ is contained in the support of $L_c(\triv)$ if and only if $c$ is not contained in any positive hyperplane $C$ such that $s=|W:W_S|_{q_c}$ vanishes on its intersection with the positive cone. 
\end{theorem} 

Our proof of Theorem \ref{support theorem1} combines Theorem \ref{ren support theorem} with an analysis of the support of $L_c(\triv)$ in the positive cone $c_{H,\chi} \geq 0$, where $\Delta_c(\triv)$ is projective. Here the Schur elements play the decisive role. Of course, one expects that the biadjunction may be chosen so that $s$ is a Laurent polynomial in the parameters, so that in practice the Schur element is zero on the whole hyperplane $C$. This is known to hold at least for the infinite family $G(r,p,n)$ and all real reflection groups, for which $s$ may be calculated using a symmetrizing form. Moreover, it is conjectured that for all complex reflection groups a symmetrizing form with properties allowing explicit calculation exist; we have used the computer algebra package CHEVIE to tabulate results in all cases assuming these conjectures.

We mention some previously known special cases of this result: first, taking $S$ to be the biggest stratum, with $W_S=1$, and assuming $W=G(r,1,n)$, we recover the description of the set of singular values obtained by Dunkl-Opdam in subsection 3.4 of \cite{DuOp}, and extend this and the previously known result for real reflection groups \cite{DJO} in a uniform fashion to all complex reflection groups; secondly, taking $W$ to be a real reflection group, we recover Etingof's calculation of the support of $L_c(\triv)$ for finite Coxeter groups, and in particular determine when it is finite dimensional, a result first obtained by Varagnolo-Vasserot \cite{VaVa} in the case of a Weyl group $W$ and equal parameters $c$.

Losev \cite{Los2} and Gerber \cite{Ger} have given combinatorial algorithms computing, for a given parameter $c$, the set of finite dimensional irreducible representations of the rational Cherednik algebra for groups in the infinite family $G(r,1,n)$, but it seems to be a difficult combinatorial problem (recently solved in some cases including the spherical case at ``integral parameters" by Gerber-Norton \cite{GeNo}) to start with a given lowest weight $E$ and use these algorithms to compute the set of $c$ for which $L_c(E)$ is finite dimensional. In particular, for those familiar with the algorithms of Losev and Gerber it might be a surprise to learn that the set of $c$ for which $L_c(\triv)$ is finite dimensional has no components of codimension greater than two for any of the groups $G(r,1,n)$ (see \ref{gr1n} for this).

\section{Notation}

\subsection{Reflection groups} Let $V$ be a finite dimensional $\CC$-vector space and let $W \subseteq \mathrm{GL}(V)$ be a \emph{complex reflection group}: a finite group of linear transformations of $V$ that is generated by the set
$$R=\{r \in W \ | \ \mathrm{codim}_V(\mathrm{fix}(r))=1 \}$$ of \emph{reflections} it contains. We write $$\mathcal{A}=\{ \mathrm{fix}(r) \ | \ r \in R \}$$ for the set of \emph{reflecting hyperplanes} for $W$, and given $H \in \mathcal{A}$ we fix a linear form $\alpha_H \in V^*$ with
$$\{ v \in V \ | \ \la \alpha_H,v \ra=0\}=H.$$ For $r \in R$ a reflection we will write $\alpha_r=\alpha_{\mathrm{fix}(r)}$.

\subsection{Dunkl operators} For each $r \in R$ let $c_r \in \CC$ be a number such that $c_r=c_{w r w^{-1}}$ for all $r \in R$ and $w \in W$. Given $y \in V$ we write $\partial_y$ for the derivative in the direction $y$ and define a \emph{Dunkl operator} by the formula
$$y(f)=\partial_y(f)-\sum_{r \in R} c_r \la \alpha_r,y \ra \frac{f-r(f)}{\alpha_r} \quad \hbox{for $f \in \CC[V]$.}$$ These operators commute with one another. We will write $\mathcal{C}$ for the parameter space, consisting of all conjugacy-invariant functions $c:r \mapsto c_r$ with values in $\CC$.

\subsection{The Euler field and the $c$-function} Fix dual bases $y_1,\dots,y_n$ and $x_1,\dots,x_n$ of $V$ and $V^*$. The \emph{Euler field} is the vector field on $V$ defined by
$$\eu=\sum_{i=1}^n x_i \partial_{y_i}.$$ In terms of the Dunkl operators it may be expressed as
$$\eu=\sum_{i=1}^n x_i y_i + \sum_{r \in R} c_r(1-r).$$ Given an irreducible $\CC W$-module $E$ we write $c_E$ for the scalar by which $\sum_{r \in R} c_r(1-r)$ acts on $E$.

\subsection{The rational Cherednik algebra} The \emph{rational Cherednik algebra} is the subalgebra $H_c=H_c(W,V)$ of $\mathrm{End}_\CC(\CC[V])$ generated by the group algebra $\CC W$, the ring $\CC[V]$ acting on itself by multiplication, and the Dunkl operators $y \in V$. The representation $\CC[V]$ of $H_c$ has a unique simple quotient that we denote by $L_c(\triv)=\CC[V]/I$ for some ideal $I$, and our aim here is to determine the zero set $V(I)=\mathrm{supp}(L_c(\triv))$.

\subsection{Category $\OO_c$} Let $\OO_c=\OO_c(W,V)$ be the category of finitely-generated $H_c$-modules on which each Dunkl operator $y$ acts locally nilpotently. This is a highest weight category with standard objects 
$$\Delta_c(E)=\mathrm{Ind}_{\CC[V^*] \rtimes W}^{H_c} (E)$$ indexed by isoclasses of irreducible representations $E$ of $\CC W$, with respect to the ordering given by $E < F$ if $c_E - c_F \in \ZZ_{>0}$. We will write $P_c(E)$ and $L_c(E)$ for the projective cover and top, respectively, of $\Delta_c(E)$. There is a formal analytic  version of $\OO_c$: the Cherednik algebra $H_c$ may be completed at $0 \in V$ to obtain an algebra $\widehat{H_c}$, generated by power series $\widehat{\CC[V]}$, Dunkl operators, and $W$, and the category $\widehat{\OO_c}$ is the category of $\widehat{H_c}$-modules that are finitely generated over $\widehat{\CC[V]}$. Completion at zero $M \mapsto \widehat{M}$ defines an equivalence $\OO_c \cong \widehat{\OO_c}$.

The definition of $\OO_c$ may be generalized as follows: given $\overline{\lambda} \in V^*/ W$, we define $\OO_{c,\overline{\lambda}}$ as the category of $H_c$-modules $M$ that are finitely generated over $\CC[V]$ and such that for each $h \in \CC[V^*]^W$ the operator $h-h(\lambda)$ is locally nilpotent on $M$.

\subsection{Reparametrization} Given $H \in \mathcal{A}$ we write $$W_H=\{w \in W \ | \ w(v)=v \quad \hbox{for all $v \in H$}\}$$ for the pointwise stabilizer of $H$ in $W$ and we let $n_H=|W_H|$ be its cardinality. Thus $W_H$ is a cyclic reflection subgroup of $W$. Given a linear character $\chi$ of $W_H$ we write
$$e_{H,\chi}=\frac{1}{n_H} \sum_{r \in W_H} \chi(r^{-1}) r$$ for the corresponding idempotent. For $r \in W_H$ we have
$$r=\sum_{\chi \in W_H^\vee} \chi(r) e_{H,\chi}$$ so the formula for the Dunkl operators can be rewritten
$$
y=\partial_y-\sum_{\substack{H \in \mathcal{A} \\ r \in W_H \\ \chi \in W_H^\vee}}  \frac{\la \alpha_H,y \ra}{\alpha_H} c_r (1-\chi(r)) e_{H,\chi}=\partial_y-\sum_{H \in \mathcal{A}} \frac{\la \alpha_H,y \ra}{\alpha_H} \sum_{\chi \in W_H^\vee} c_{H,\chi} n_H e_{H,\chi}$$ where for each linear character $\chi$ of $W_H$ we define
$$c_{H,\chi}=\frac{1}{n_H} \sum_{r \in W_H} c_r(1-\chi(r)).$$ We observe that $c_{H,1}=0$ and that since $\chi(1)=1$ the sum is effectively over the non-identity elements of $W_H$. We will use the symbol $c$ to denote a parameter for the rational Cherednik algebra, whether in the coordinates $r$ or the coordinates $(H,\chi)$. Note that if $n_H=2$ then there is a unique non-identity character $\chi$ of $W_H$ and a unique non-identity element $r_H$ of $W_H$ and we have $c_{H,\chi}=c_{r_H}$. 

We have
$$\sum_{r \in R} c_r (1-r)=\sum_{H \in \mathcal{A}} \sum_{r \in W_H} c_r \sum_{\chi \in W_H^\vee} (1-\chi(r)) e_{H,\chi}=\sum_{H \in \mathcal{A}} \sum_{\chi \in W_H^\vee} c_{H,\chi} n_H e_{H,\chi},$$ and it follows that $c_E \geq 0$ for all irreducible representations $E$ of $\CC W$ provided that $c_{H,\chi} \geq 0$ for all $(H,\chi)$. In particular, the polynomial representation $\CC[V]=\Delta_c(\triv)$ is projective if $c_{H,\chi} \geq 0$ for all $(H,\chi)$.

\subsection{The braid group and the finite Hecke algebra} The \emph{braid group} $B=B_W$ of $W$ is the fundamental group $B_W=\pi_1(V^\circ/W,*)$ of the space $V^\circ/W$. It is generated by the monodromy operators $T_H$ for $H \in \mathcal{A}$. We define
$$q_{H,\chi}=e^{2 \pi i c_{H,\chi}}$$ and write $q=q_c$ for the collection of all these parameters $q_{H,\chi}$. The \emph{finite Hecke algebra} $\mathcal{H}_c$ is the quotient of the group algebra $\CC B$ by the relations
$$\prod_{\chi \in W_H^\vee} (T_H-\chi(r_H) q_{H,\chi})=0,$$ where we write $r_H \in W_H$ for the generator achieving a positive rotation around $H$ by the angle $2 \pi / n_H$. Note that since $q_{H,1}=1$ for all $H$, the one-dimensional $\CC B$-module $\CC$ (on which each $T_H$ acts trivially) factors through $\mathcal{H}_c$.

Given a parabolic subgroup $W_S$ of $W$, there is an inclusion of the corresponding finite Hecke algebra $\mathcal{H}_{S,c} \subseteq \mathcal{H}_c$ described in section 2D of \cite{BMR}, so we have well-defined induction and restriction functors. The paper \cite{GGOR} constructs a functor $$\mathrm{KZ}:\OO_c \rightarrow \mathcal{H}_c \mathrm{-mod}$$ making $\OO_c$ into a highest weight cover of  the category of $\mathcal{H}_c$-modules. In particular, $\mathrm{KZ}$ is fully faithful on projective objects. 

We will use the symbol $\q_{H,\chi}$ for a formal variable, in contrast to the numbers $q_{H,\chi}$. The Hecke algebra $\mathcal{H}$ for these generic parameters will be denoted without the subscript $c$. For the trivial stratum $S$ with $\mathcal{H}_S=P$, where $P=\ZZ[\q_{H,\chi}^{\pm 1}]$ is the parameter ring for the Hecke algebra, Etingof \cite{Eti2} has recently observed that $\mathcal{H}$ is free over $P$ as a consequence of work of Losev \cite{Los3} and Marin-Pfeiffer \cite{MaPf}. We will freely use this fact in the rest of this paper.

\section{Support of $L_c(\triv)$}

\subsection{Biadjointness and relative Schur elements} \label{biadjointness} In this section we will make crucial use of the parabolic induction and restriction functors of \cite{BeEt}, and the results of \cite{Sha} and \cite{Los1} asserting that they are biadjoint. Let $p \in V$ be a point and let $W_p=\{w \in W \ | \ w(p)=p \}$ be the parabolic subgroup of $W$ fixing $p$. Write $\OO_{c,p}$ for the category $\OO$ of the rational Cherednik algebra $H_c(W_p,U)$, where $U$ is the (unique) $W_p$-stable complement to the fixed space $V^{W_p}$, and write $\mathcal{H}_{p,c}$ for the finite Hecke algebra of $W_p$ at the parameter $c$, which is a subalgebra of $\mathcal{H}_c$. The following theorem summarizes the work of these authors that we shall need.
\begin{theorem}
\begin{enumerate}
\item[(a)] The functors $\mathrm{Res}_p$ and $\mathrm{Ind}_p$ are biadjoint, and $p \in \mathrm{supp}(L_c(E))$ if and only if $$\mathrm{Res}_p(L_c(E)) \neq 0.$$
\item[(b)] The KZ functor intertwines $\mathrm{Res}_p$ with $\mathrm{Res}^{\mathcal{H}_c}_{\mathcal{H}_{p,c}}$ and intertwines $\mathrm{Ind}_p$ with $\mathrm{Ind}^{\mathcal{H}_c}_{\mathcal{H}_{p,c}}$. 
\item[(c)] The functors $\mathrm{Res}^{\mathcal{H}_c}_{\mathcal{H}_{p,c}}$ and $\mathrm{Ind}^{\mathcal{H}_c}_{\mathcal{H}_{p,c}}$ are biadjoint.
\end{enumerate}
\end{theorem} \begin{proof} Part (a) follows from the definintions in \cite{BeEt} and the work \cite{Los1}. For (b) and (c) we argue as follows: the KZ functor is a Serre quotient functor (see section 5 of \cite{GGOR}), and it follows from the definitions that parabolic induction and restriction are exact functors respecting the subcategories of objects killed by KZ, so they descend to functors between the Serre quotient categories. Moreover, the biadjunction $(\mathrm{Res}_p,\mathrm{Ind}_p)$ induces a biadjoint pair on these functors between quotient categories. But \cite{Sha} Theorem 2.1 shows that the functor induced by $\mathrm{Res}_p$ is ordinary restriction for finite Hecke algebras, and it follows that the functor induced by $\mathrm{Ind}_p$ is isomorphic to ordinary coinduction (and ordinary coinduction, which is therefore isomorphic to induction) for finite Hecke algebras. This proves (b) and (c).
\end{proof} By taking a unit-counit composition, the biadjunction $(\mathrm{Res}^{\mathcal{H}_c}_{\mathcal{H}_{p,c}},\mathrm{Ind}^{\mathcal{H}_c}_{\mathcal{H}_{p,c}})$ produces an endomorphism of the identity functor of $\mathcal{H}_c$-mod
$$1 \longrightarrow \mathrm{Ind}^{\mathcal{H}_c}_{\mathcal{H}_{p,c}} \mathrm{Res}^{\mathcal{H}_c}_{\mathcal{H}_{p,c}} \longrightarrow 1.$$ When applied to any $\mathcal{H}_c$-module $M$ with $\mathrm{End}_{\mathcal{H}_c}(M)=\CC$ this gives a scalar $s_M(c)$ which we refer to as the \emph{relative Schur element for $M$}. In particular, the \emph{principal relative Schur element} is the scalar corresponding to the trivial module $s(c)=s_\CC(c)$. Since $\mathcal{H}_c$ only depends on the exponential $q_c$ of $c$ we may also write $s(q_c)$ or $|W:W_p|_{q_c}$ for this number if we wish to emphasize the dependence on $W_p$. 

We note that there is an ambiguity here, corresponding to the particular choice of an isomorphism between induction and coinduction. However, this ambiguity only affects $s(c)$ up to multiplication by a non-zero number, so does not affect the statements or proofs of our results. In specific calculations we use the normalization coming from a choice of trace form on the Hecke algebra, which is not known to exist in general.

\subsection{Induction, restriction, and the support of $L_c(E)$.}

\begin{lemma} \label{support lemma} Let $L$ be an irreducible object of $\OO_c$ with standard cover $\Delta$ and projective cover $P$ and let $S$ be a stratum and $p \in S$. The following are equivalent:
\begin{itemize}
 \item[(a)] $S \subseteq \mathrm{supp}(L)$.
 \item[(b)] The counit of the adjunction $\mathrm{Ind}_p( \mathrm{Res}_p( \Delta)) \longrightarrow \Delta$ is surjective.
 \item[(c)] The counit of the adjunction $\mathrm{Ind}_p( \mathrm{Res}_p( P)) \longrightarrow P$ is surjective.
 \item[(d)] $P$ is a summand of $\mathrm{Ind}_p (\mathrm{Res}_p(P))$.
\item[(e)] $\mathrm{KZ}(P)$ is a summand of its induction-restriction $\mathrm{Ind}^{\mathcal{H}_c}_{\mathcal{H}_{p,c}}(\mathrm{Res}^{\mathcal{H}_c}_{\mathcal{H}_{p,c}}(\mathrm{KZ}(P)))$.
 \end{itemize}
\end{lemma}
\begin{proof} The equivalences (a) $\iff$ (b) $\iff$ (c) follow from the facts that $S$ is in the support of a module if and only if the restriction to $p$ is non-zero, and that $L$ is the top of $\Delta$ and $P$. Since $P$ is projective, (c) implies (d). The following observation shows (d) implies (a): if $P$ is a summand of $\mathrm{Ind}_p (\mathrm{Res}_p (P))$ then $L$ is a quotient of a module induced from $S$ and hence $S \subseteq \mathrm{supp}(L)$.  Finally we check the equivalence of (d) and (e). By Theorem 5.3 of \cite{GGOR}, KZ is fully faithful on projective objects. Since it intertwines $\mathrm{Ind}_p$ and $\mathrm{Res}_p$ with $\mathrm{Ind}^{\mathcal{H}_c}_{\mathcal{H}_{p,c}}$ and $\mathrm{Res}^{\mathcal{H}_c}_{\mathcal{H}_{p,c}}$ we see that (d) and (e) are equivalent.
\end{proof} In the special case where $\mathrm{KZ}(P)$ is one-dimensional, we can improve this criterion:
\begin{corollary} \label{rank one}
With the hypotheses of Lemma \ref{support lemma}, suppose that $\mathrm{KZ}(P)$ is one-dimensional. Then $S \subseteq \mathrm{supp}(L)$ if and only if $s_{\mathrm{KZ}(P)}(q_c) \neq 0$. 
\end{corollary}
\begin{proof}
We abbreviate $M=\mathrm{KZ}(P)$. If $|W:W_p|_{q_c} \neq 0$ then by definition the counit of adjunction is, up to a non-zero scalar, a right inverse to the unit of adjunction for $M$, so $M$ is a summand of its induction-restriction (here we do not use our hypothesis on the dimension of $M$)  and we may apply Lemma \ref{support lemma}. Conversely, suppose $M$ is a summand of its induction restriction. Observe that by the adjunction and our hypothesis, 
$$\mathrm{Hom}\left(\mathrm{Ind}^{\mathcal{H}_c}_{\mathcal{H}_{p,c}}(\mathrm{Res}^{\mathcal{H}_c}_{\mathcal{H}_{p,c}}(M)),M \right) \cong \mathrm{End}(\mathrm{Res}^{\mathcal{H}_c}_{\mathcal{H}_{p,c}}(M) \cong \CC.$$ This implies that the projection witnessing $M$ as a summand is a non-zero multiple of the counit of adjunction. Similarly, the injection witnessing $M$ as a summand is a non-zero multiple of the unit. This implies that the unit-counit composition is non-zero on $M$, or in other words that $s_{\mathrm{KZ}(P)}(q_c) \neq 0$.
\end{proof}

\subsection{The positive cone} \label{cone subsection}

\begin{lemma} \label{positive cone}
Suppose the parameters $c_{H,\chi} \geq 0$ are all non-negative real numbers. Then we have $S \subseteq \mathrm{supp}(L_c(\triv))$ if and only if $|W:W_p|_{q_c} \neq 0$. 
\end{lemma}
\begin{proof}
Since the parameters are non-negative, we have $c_E \geq 0$ for all representations $E$ of $W$, and it follows that $\Delta_c(\triv)$ is projective. Applying Corollary \ref{rank one} now implies the result.
\end{proof}

\subsection{Counit of adjunction} We now switch gears and study the counit of adjunction for standard modules in detail; our aim is to describe it explicitly in terms of a joint eigenfunction for the Dunkl operators. Fix a stratum $S$ and a point $\lambda \in V^*$ with $W_\lambda=W_S$. Proposition 3.13 of \cite{BeEt} together with \cite{Los1}, sections 3 and 4, implies that the parabolic induction and restriction functors may be described as follows: the restriction functor applied to a module $M$ in $\OO_c$ is the composite of the completion at $0 \in V$ functor $M \mapsto \widehat{M}$, which is an equivalence $\OO_c \cong \widehat{\OO_c}$, the functor $E_{\overline{\lambda}}:\widehat{O_c} \rightarrow \OO_{c,\overline{\lambda}}$ given by
$$E_{\overline{\lambda}}(M)=\{f \in M \ | \ \hbox{for all $h \in \CC[V^*]^W$ there is $N>0$ with $(h-h(\lambda))^N f=0$} \},$$ which depends only on the $W$-orbit ${\overline{\lambda}}=W \lambda$, and an equivalence $\psi_\lambda:\OO_{c,\overline{\lambda}} \rightarrow \OO_c(W_\lambda,V/V^{W_\lambda})$ given by
$$\psi_\lambda(M)=\{f \in M_\lambda \ \vert \  y f=\lambda(y)f \quad \hbox{for all $y \in V^{W_\lambda}$} \},$$ where we write $$M_\lambda=\{f \in M \ \vert \ \hbox{for each $y \in V$ there is $N$ with $(y-\lambda(y))^N f=0$} \}.$$ The induction functor is the composite of the inverse equivalence $\psi_\lambda^{-1}$, the completion at zero functor $\widehat{\cdot}$, which is left adjoint to $E_\lambda$, and the inverse equivalence $E_0$ to the completion at zero functor. Thus the counit for $(\mathrm{Ind},\mathrm{Res})$ is surjective on a module $M \in \OO_c$ if and only if the counit 
\begin{align*}
\CC[[V]] &\otimes_{\CC[V]} E_\lambda(\widehat{M}) \longrightarrow \widehat{M} \\
&f \otimes m  \longmapsto  fm
\end{align*}
is surjective. By Nakayama's lemma, this may be tested upon taking the $0$-fiber. Thus for example the counit for $\Delta_c(E)$ is surjective if and only if there is $f \in \widehat{\Delta_c(E)}$ with $$f(0) \neq 0 \quad \text{and} \quad (h-h(\lambda))^N f=0 \quad \hbox{for some $N>0$ and all $h \in \CC[V^*]^W$.}$$ Moreover, since $w M_\lambda=M_{w \lambda}$ and 
$$E_{\overline{\lambda}}(M)=\bigoplus_{\lambda \in \overline{\lambda}} M_\lambda,$$ the counit for $\Delta_c(E)$ is surjective if and only if there is $f \in \widehat{\Delta_c(E)}$ with $$f(0) \neq 0 \quad \text{and} \quad (y-\lambda(y))^N f=0 \quad \hbox{for some $N>0$ and all $y \in V$.}$$ We will therefore construct particular elements of $\widehat{\Delta_c(E)}_\lambda$ in the next subsection.

\subsection{The Dunkl-de Rham complex and $(W,E)$-exponential functions}

We will write $\CC(\cc)$ for the quotient field of the ring $\CC[\cc]=\CC[\mathcal{C}]$ of polynomial functions on the parameter space $\mathcal{C}$, where in we use the symbol $\cc$ for a formal parameter, and reserve $c$ for its specialization to a number. For each non-negative integer $d$, the Dunkl operators define a non-degenerate pairing
$$(\CC(\cc)[V^*]^d \otimes E^*) \otimes (\CC(\cc)[V]^d \otimes E) \longrightarrow \CC(\cc)$$ given on decomposable vectors by the formula
$$(f \otimes \phi) \otimes (g \otimes e) \mapsto \phi(f \cdot (g \otimes e)) \quad \hbox{for $f \in \CC[V^*]^d$, $\phi \in E^*$, $g \in \CC[V]^d$, and $e \in E$,}$$ where $f \cdot (g \otimes e)$ denotes the action of $f \in H_\cc$ on $g \otimes e \in \Delta_\cc(E)$; since $f$ and $g$ have the same degree this produces a vector in $E$ on which we may then evaluate $\phi$. For each non-negative integer $d$, let $f_{d,I,j} \in \CC(\cc)[V]^d \otimes E$ and $f_{d,I,j}^\vee \in \CC(\cc)[V^*] \otimes E^*$ be dual bases with respect to this pairing (where $I$ runs over an indexing set for a basis for the space of polynomials of degree $d$ and $j$ runs over an indexing set for a basis of $E$). The \emph{$(W,E)$-exponential function} is the formal series (a power series at the origin in $V^* \times V$ with coefficients that are rational functions on the parameter space $\mathcal{C}$)
$$\mathrm{exp}_E=\sum_{d,I,j} f_{d,I,j} f_{d,I,j}^\vee,$$ which is independent of the bases chosen. We shall see that it converges uniformly on compact subsets of the complement $\mathcal{C}_E$ of a certain hyperplane arrangement in $\mathcal{C}$,
$$\mathcal{C}_E=\mathcal{C} \setminus \bigcup_{\substack{F \in \mathrm{Irr}(\CC W) \\ d \in \ZZ_{>0}}} \{d=c_F-c_E \}$$ and it therefore defines a holomorphic function on $V^* \times V \times \mathcal{C}_E$ with values in $E^* \otimes E=\mathrm{End}(E)$. 

Our first aim here is to give a bound on the denominators that can appear in $\mathrm{exp}_E$. The tool we need for this is the \emph{Dunkl-de Rham complex}, introduced in \cite{DuOp} in the case $E=\triv$, $$K^\bullet=\CC(\cc)[V] \otimes E \otimes \Lambda^\bullet V^*,$$ with $W$-equivariant differential $d_\cc$ defined by the formula
$$d_\cc(f \otimes \omega)=\sum_i y_i f \otimes dx_i \wedge \omega$$ for any choice of dual bases $y_1,\dots,y_n$ of $V$ and $x_1,\dots,x_n$ of $V^*$. Here we write $dx_i$ in place of $x_i$ simply as a formal device to distinguish elements of $\Lambda V^*$ from elements of $\CC[V]$. We write
$$\partial(f \otimes dx_{i_1} \cdots dx_{i_p})=\sum_i (-1)^{j+1} x_{i_j} f \otimes dx_{i_1} \cdots \widehat{dx_{i_j}} \cdots dx_{i_p} \quad \hbox{for $f \in \CC(\cc)[V] \otimes E$}$$ for the usual Koszul differential on the complex $K^\bullet$.

\begin{lemma} \label{derham lemma}
\begin{itemize}
\item[(a)] We have $$d_\cc \partial+\partial d_\cc=d+p+\cc_E-\cc_F$$ as operators on the $F$-isotypic component of the polynomial degree $d$ piece $$(\CC(\cc)[V]^d \otimes E \otimes \Lambda^p V^*)_F$$ of $K^p$.
\item[(b)] For each $d$ we have
$$
\prod_{\substack{E \in \mathrm{Irr}(\CC W) \\ 1 \leq m \leq d}} (m+\cc_E-\cc_F) \sum_I f_{d,I}^\vee f_{d,I} \in \CC[\cc][V].$$ for any choice of dual bases.
\item[(c)] The series $\mathrm{exp}_E$ converges, uniformly on compact subsets of $\mathcal{C}_E$, to a holomorphic function on $V \times V^*$ with values in $E^* \otimes E$. 
\end{itemize}
\end{lemma}
\begin{proof}
Part (a) is proved by the same calculation as in Proposition 2.4 of \cite{DuOp}. We prove (b) and (c). Fix a bases $y_1,\dots,y_n$ of $V$ and $e_1,\dots,e_m$ of $E$ with dual bases $x_1,\dots,x_n$ of $V^*$ and $\phi_1,\dots,\phi_m$ of $E^*$. For each $1 \leq j \leq m$ and each multi-index $I \in \ZZ_{\geq 0}^n$ let $f_{I,j}^\vee=y^I \phi_j$ (during this proof we suppress the index $d$). We will construct the dual basis $f_{I,j}$ of $\CC(\cc)[V] \otimes E$ inductively starting with $f_{0,j}=e_j$ for $1 \leq j \leq m$. Fixing a multi-index $I$ of degree $d>0$ and some $1 \leq j \leq m$ and inductively assume we have constructed the dual basis elements $f_{I-\epsilon_i,j}$ for all $1 \leq i \leq n$, where $\epsilon_i$ denotes the multi-index with $1$ in the $i$th position and zeros elsewhere, and we have $f_{I,j}=0$ by convention if some entry of $I$ is negative. By choice of $f_{I,j}^\vee$ it suffices to find $f_{I,j}$ with
$$y_i f_{I,j}=f_{I-\epsilon_i,j} \quad \hbox{for all $1 \leq i \leq n$.}$$ We put
$$\omega=\sum f_{I-\epsilon_i,j} dx_i$$ and observe that by our inductive hypothesis we have $d_\cc \omega=0$. For each $F \in \mathrm{Irr}(\CC W)$ write $e_F \in \CC W$ for the corresponding central idempotent, which acts on a $\CC W$-module as the projection onto the $F$-isotypic component. Since $d_\cc$ commutes with $\CC W$, defining $\omega_F=e_F \omega$ we have $d_\cc \omega_F=0$. Thus using (a)
$$d_\cc \partial \omega_F=(d+\cc_E-\cc_F) \omega_F \quad \implies \quad \omega=d_\cc \left(\sum_F \frac{1}{d+\cc_E-\cc_F} \partial \omega_F \right).$$ Thus we may put 
\begin{equation} \label{recursion} f_{I,j}=\sum_F \frac{1}{d+\cc_E-\cc_F} \partial \omega_F=\sum_{F,i} \frac{1}{d+\cc_E-\cc_F} e_F( x_i f_{I-\epsilon_i,j}). \end{equation} This establishes (b).

Now assume $x_1,\dots,x_n$ and $e_1,\dots,e_m$ are an orthonormal bases of $V^*$ and $E$ with respect to $W$-invariant Hermitian forms. Then 
$$\frac{1}{\sqrt{I!}} x^I e_j \quad \hbox{for $I$ of degree $d$ and $1 \leq j \leq m$, where $I!=i_1 ! i_2 ! \cdots i_n!,$}$$ is an orthonormal basis of $\CC[V]^d \otimes E$ with respect to the positive definite Hermitian form induced by non-deformed partial differentiation. For $f \in \CC[V]^d \otimes E$ we write $|f|_\infty$ for the supremum norm with respect to this basis, equal to the maximum absolute value of a coefficient of $f$ on the basis. Since $e_F$ is an orthogonal projection for the inner product with respect to which this basis is orthonormal, we have the bound $|e_F|_\infty \leq 1$ for the operator norm of $e_F$ with respect to the supremum norm on this basis. Moreover for $f \in \CC[V]^{d-1} \otimes E$ and $1 \leq i \leq n$ we have $|x_i f|_\infty \leq \sqrt{d} |f|_\infty$. Fix a compact subset $K$ of $\mathcal{C}_E$ and let $k$ be the maximum value of $|c_F-c_E|$ on $K$ over all $F$ and  let $d>k$. Putting all this together with \eqref{recursion} gives
$$|f_{I,j}|_\infty \leq \frac{n |\mathrm{Irr}(\CC W)| \sqrt{d}}{d-k} M_{d-1}, $$ where  $M_{d-1}$ is the maximum value of $|f_{I',j}|_\infty$ for $I'$ of degree $d-1$. This proves (c).

\end{proof}

\subsection{The renormalized $W$-exponential functions} \label{BA}

We will assume from now on that $E=\triv$ is the trivial representation, abbreviate $e=\mathrm{exp}_\triv$, and define a certain renormalization of a specialized version of $e$. This renormalization depends in a crucial way on the specialization to a point $\lambda \in V^*$. If $y \neq 0$ we may extend $y=y_1$ to a basis $y_1,\dots,y_n$ of $V$ and fix a monomial basis $f_{d,I}^\vee=y^I$ as above. We then have $y(f_{d,I})=0$ if $I_1=0$ and $y(f_{d,I})=f_{d-1,I-\epsilon_1}$ if $I_1>0$, as in the proof of Lemma \ref{derham lemma}. It follows that, if we define $$y \cdot e=\sum f_{d,I}^\vee y(f_{d,I})$$ by using the action of the Dunkl operators on the $x$-part of $e$, then
\begin{equation} \label{eig equation}
y \cdot e=\sum_{d,I} y f_{d,I}^\vee f_{d,I}=ye.
\end{equation}

Given $\lambda \in V^*$ we put
$$e_\lambda(x,c)=\sum_{d,i} f_{d,I}^\vee(\lambda) f_{d,I},$$ and note that by \eqref{eig equation} we have
$$y \cdot e_\lambda=\lambda(y) e_\lambda \quad \text{and} \quad w e_\lambda=e_{w \lambda} \quad \hbox{for all $y \in V$ and $w \in W$,}$$ so that in particular $e_\lambda$ is fixed by $W_\lambda$. 

We choose an entire analytic function $F_\lambda(c)$ on $\mathcal{C}$ with the properties that 
\begin{enumerate}
\item[(1)] For each $d$, 
$$F_\lambda(c) \sum_i f_{d,i}^\vee(\lambda) f_{d,i}$$ is an analytic function of $c$, and 
\item[(2)] for each $c \in \mathcal{C}$, there is some $d$ with 
$$F_\lambda(c) \sum_i f_{d,i}^\vee(\lambda) f_{d,i} \neq 0.$$
\end{enumerate} The existence of $F_\lambda$ is guaranteed by the bound on the denominators given in part (b) of Lemma \ref{derham lemma}. The set of zeros of $F_\lambda(c)$ is precisely the union of the zero sets of the least common denominators of the sums $\sum_i f_{d,i}^\vee(\lambda) f_{d,i}$, and in particular it consists of certain of the hyperplanes defined by the equations $\cc_E-m=0$, for $m$ a positive integer and $E$ an irrep of $\CC W$. These are positive hyperplanes in the sense defined in the introduction.

We define the \emph{renormalized specialized $W$-exponential function} $q_\lambda(x,c)$ by
$$q_\lambda(x,c)=F_\lambda(c) e_\lambda(x,c).$$ It is a formal power series on $V$ whose coefficients are analytic functions of $c$ without common zeros. So for any fixed $c$, the function $q_\lambda=q_\lambda(x,c)$ is a non-zero element of $\widehat{\CC[V]}$ satisfying
$$y q_\lambda=\lambda(y) q_\lambda \quad \hbox{for all $y \in V$} \quad \text{and} \quad w q_\lambda=q_\lambda \quad \hbox{for all $w \in W_\lambda$.}$$

Note that its definition depends on our choice of $F_\lambda$, or in other words, it is really only well-defined up to multiplication by a function of the form $\mathrm{exp}(f)$, for $f$ an analytic function on $\mathcal{C}$; for a fixed $c$, it is well-defined up to multiplication by a non-zero scalar. This ambiguity does not affect the arguments in the remainder of the paper, but would be very interesting to resolve in some natural fashion.

\begin{theorem} \label{sol theorem} Fix a stratum $S$ and let $\lambda \in V^*$ be a point with $W_\lambda=W_S$. The stratum $S$ is in the support of $L_c(\triv)$ if and only if $F_\lambda(c) \neq 0$. In particular, the zero set of $F_\lambda$ depends only on $S$ and hence we may choose $F_S=F_\lambda$ to be independent of $\lambda \in S$.
\end{theorem} 
\begin{proof}
We write $\mathrm{Res}=\mathrm{Res}_p$ for the restriction functor corresponding to some point $p \in S$ and fix an isomorphism $\mathrm{Res} \cong \mathrm{res}_\lambda$ whose existence is guaranteed by \cite{Los1}. By proposition 1.9 of \cite{Sha}, the module $\mathrm{Res}(\Delta_c(\triv))$ has a standard filtration, and by Proposition 3.14 of \cite{BeEt} we therefore have $\mathrm{Res}(\Delta_c(\triv)) \cong \Delta_c(W_S,\triv)$. Now examining the definition of $\mathrm{res}_\lambda$ shows that 
$$\Delta_c(W_S,\triv) \cong \mathrm{res}_\lambda(\Delta_c(\triv))=\{f \in \widehat{\CC[V]}_\lambda \ \vert \ yf=\lambda(y)f \quad \hbox{for all $y \in V^{W_S}$}\}$$ and in particular $q_\lambda$ is characterized up to scalars as the unique non-zero $W_S$-fixed element of this space with $y q_\lambda=\lambda(y) q_\lambda$ (note here that there is an origin shift in the definition of $\mathrm{res}$: the Dunkl operators for $W_S$ act as $y-\lambda(y)$ and are thus locally nilpotent), and so $\mathrm{res}_\lambda(\Delta_c(\triv))$ is generated by $q_\lambda$ as a $\CC[V/V^{W_\lambda}]$-module. Formula (3) from subsection 2.1 of \cite{Los1} shows that the constant term of every other element of $\widehat{\CC[V]}_\lambda$ is a multiple of $F_\lambda(c)$ and this proves the lemma.
\end{proof}

As an example, suppose $W=\{1,-1 \}$ is the group of order two acting on $V=\CC$. Direct calculation with Dunkl operators shows that the $W$-exponential function is
\begin{align*}
e(x,y,c)=1+&\sum_{n=1}^\infty \frac{y^{2n} x^{2n}}{(1-2c) \cdot 2 \cdot (3-2c) \cdot 4 \cdots (2n-1-2c) \cdot 2n} + \\
&\sum_{n=0}^\infty \frac{y^{2n+1} x^{2n+1}}{(1-2c) \cdot 2 \cdot (3-2c) \cdot 4 \cdots 2n \cdot (2n+1-2c)} 
\end{align*} so that for $\lambda \neq 0$ we may take 
$$F_\lambda(c)=\frac{\Gamma(1-c)}{\Gamma(1-2c)}.$$

\subsection{Support theorem} Finally we can complete the calculation of the support of $L_c$. We recall that a \emph{positive} hyperplane is a hyperplane $C \subseteq \mathcal{C}$ of the form
$$C=\left\{\sum a_{H,\chi} c_{H,\chi}=a \right\} \quad \hbox{with $a_{H,\chi} \geq 0$ and $a >0$.}$$ Note that if $C$ is a positive hyperplane, then the set of points of $C$ with non-negative real coordinates is Zariski-dense in $C$, and in particular a positive hyperplane is uniquely determined by its set of points with non-negative real coordinates.
\begin{theorem} \label{support theorem2}
Let $S$ be a stratum with relative principal Schur element $s=|W:W_S|_{q_c}$. Then $S$ is contained in the support of $L_c(\triv)$ if and only if $c$ is not contained in any positive hyperplane $C$ such that $s$ vanishes on intersection of $C$ with the positive cone.
\end{theorem}

\begin{proof}
By Theorem \ref{sol theorem} the set of $c$ for which $S$ is not contained in the support of $L_c(\triv)$ is a union of positive hyperplanes. Let $C$ be such a hyperplane. To points in the intersection of this hyperplane with the positive cone we may apply Lemma \ref{positive cone}, and this implies that $s$ vanishes there. Conversely, given a positive hyperplane $C$ such that $s$ vanishes on its positive part, we may again apply Lemma \ref{positive cone} to conclude that $S$ is not contained in the support of $L_c(\triv)$ for $c$ in the intersection of $C$ with the positive cone. Thus the intersection of $C$ with the positive cone is contained in union of hyperplanes provided by Theorem \ref{sol theorem}, and this implies that $C$ is one of these hyperplanes. 
 \end{proof}

\subsection{Symmetrizing trace} \label{symmetry}

In this section we observe that a symmetrizing trace on the Hecke algebra, compatible with parabolic subalgebras in a sense to be made precise below, allows us to compute (up to a non-zero multiple) the principal relative Schur elements appearing in our theorem. We write $P=\CC[\q_{H,\chi}^{\pm 1}]$ for the parameter ring of the Hecke algebra; in this context we use the symbol $\q$ for a formal parameter and reserve $q$ for its specialization to a number. In order to apply the results of this subsection to the Hecke algebra, we make the following hypothesis, which will allow us to compute the relative Schur elements in terms of a symmetrizing trace.

\begin{hypothesis} \label{hyp} The Hecke algebra $\mathcal{H}$ is a projective left $\mathcal{H}_S$-module and there exists an $P$-linear map $t:\mathcal{H} \longrightarrow P$ with the following properties:
\begin{itemize}
\item[(a)] $t(ab)=t(ba)$ for all $a,b \in \mathcal{H}$ and $a\mapsto (b \mapsto t(ab))$ defines an isomorphism of $\mathcal{H}$ onto $\mathrm{Hom}_P(\mathcal{H},P)$. 
\item[(b)] We have $\mathcal{H}=\mathcal{H}_S \oplus \mathcal{H}_S^\perp$, where $\mathcal{H}_S^\perp$ denotes the orthogonal complement with respect to the trace pairing $(a,b)=t(ab)$.
\end{itemize}
\end{hypothesis} For real reflection groups this is known thanks to classical facts about finite Coxeter groups. For the groups $G(r,1,n)$ the trace form was constructed in Theorem 2.8 of \cite{BrMa}. 

Let $R$ be a commutative ring and let $A$ be an $R$-algebra, free of finite rank as an $R$-module. Let $t:A \rightarrow R$ be an $R$-linear map with $t(ab)=t(ba)$ for all $a,b \in A$. For $a \in A$ define an $R$-linear form $\phi_a:A \rightarrow R$ by $\phi_a(b)=t(ab)$, and assume that $a \mapsto \phi_a$ is an isomorphism of $A$ onto $\mathrm{Hom}_R(A,R)$. Thus $A$ is a \emph{symmetric algebra}. 

Assume that we are given an $R$-subalgebra $B \subseteq A$ such that $A$ is a projective $B$-module, and moreover $A$ is the direct sum
$A=B \oplus  B^\perp$, where $$B^\perp=\{a \in A \ | \ t(ab)=0 \ \hbox{for all $b \in B$} \}.$$ Let $\pi:A \longrightarrow B$ be the projection for this direct sum decomposition. We note that since $1 \in B$ we have $t(B^\perp)=0$, and hence $t(\pi(a))=t(a)$ for all $a \in A$. The map $a \mapsto \psi_a$, where $\psi_a(a')=\pi(a'a)$, is an isomorphism of $A$-$B$-bimodules
$$A \longrightarrow \mathrm{Hom}_B(A,B).$$ 

Let $M$ be a $B$-module. Since $A$ is a projective $B$-module of finite rank, there is a natural isomorphism
$$\mathrm{Hom}_B(A,B) \otimes_B M \longrightarrow \mathrm{Hom}_B(A,M)$$ given by $\phi \otimes m \mapsto (a \mapsto \phi(a) m)$, and we obtain an isomorphism
$$\mathrm{Hom}_B(A,M) \cong \mathrm{Hom}_B(A,B) \otimes_B M \cong A \otimes_B M$$ between the coinduction $\mathrm{Hom}_B(A,M)$ and the induction $A \otimes_B M$. 

\begin{lemma}
Assume $R$ is an integral domain and that $M$ is a free $R$-module of rank one on which $A$ acts by a linear character $\chi:A \longrightarrow R$. Define $a_\chi \in A$ and $b_\chi \in B$ by the rules
$$\chi(a)=t(a a_\chi) \quad \text{and} \quad \chi(b)=t(b b_\chi) \quad \hbox{for all $a \in A$ and $b \in B$.}$$ Then $\chi(b_\chi)$ divides $\chi(a_\chi)$ in $R$ and the composite of the maps $$M \longrightarrow \mathrm{Hom}_B(A,M) \cong A \otimes_B M \longrightarrow M$$ is given by
$$m \longmapsto \frac{\chi(a_\chi)}{\chi(b_\chi)}m.$$
\end{lemma}
\begin{proof}
From $t \circ \pi=t$ and using the fact that $\pi$ is a $B$-bimodule map, we obtain $$\chi(\pi(a a_\chi))=t(\pi(a a_\chi) b_\chi)=t(\pi(a a_\chi b_\chi))=t(a a_\chi b_\chi)=t(b_\chi a a_\chi)=\chi(b_\chi a)=\chi(b_\chi) \chi(a)$$  for all $a \in A$. It follows that if $m \in M$ is an element of $M$, and defining $\phi: a \mapsto \pi(a a_\chi)$, the element $\phi \otimes m$ corresponds, via the isomorphism $\mathrm{Hom}_B(A,B) \otimes_B M \cong \mathrm{Hom}_B(A,M)$, to the map $a \mapsto \chi(b_\chi) a m$. Writing $\psi$ for the unit-counit composition in the statement of the lemma we then have
$$\chi(b_\chi) \psi(m)=\chi(a_\chi) m \quad \hbox{for all $m \in M$.}$$ Since $M$ is free of rank one as an $R$ module, this implies the lemma.
\end{proof}

With the preceding hypotheses, we will refer to the element $\chi(a_\chi)$ as the \emph{Schur element} of the $A$-module $M$, and likewise $\chi(b_\chi)$ is the Schur element of $M$ regarded as a $B$-module.

\begin{corollary}
With the assumptions of the previous lemma, assume moreover that $m \subseteq R$ is a maximal ideal, and given an $R$-module $N$ write $N(m)=N/mN$ for its fiber at $m$. Then $M(m)$ is a summand of $\mathrm{Ind}^{A(m)}_{B(m)}(\mathrm{Res}^{A(m)}_{B(m)}(M(m)))$ if and only if $\chi(a_\chi) / \chi(b_\chi) \neq 0 \ \mathrm{mod} \ m$.
\end{corollary} 

\subsection{Coxeter groups} Suppose that $W \subseteq \mathrm{GL}(V)$ is a real reflection group; in this case we may choose a set $S$ of simple reflections such that $(W,S)$ is a Coxeter system, and the monodromy relations for the Hecke algebra are
$$(T_s-1)(T_s+q_s)=0 \quad \hbox{for $s \in S$,}$$ where we write $T_H=T_s$ and $q_s=q_{H,\chi}$ if $s$ is a  reflection about the hyperplane $H$ and $\chi$ is non-trivial. Given an expression $w=s_1 \cdots s_p$ of minimal length we write $T_w=T_{s_1} \cdots T_{s_p}$. It is a classical fact that the Hecke algebra $\mathcal{H}_c$ has $\CC$-basis $T_w$ for $w \in W$, and the formula
$$t\left(\sum a_w T_w\right)=a_1$$ defines a trace form. One checks that defining
$$\epsilon=\sum q_w^{-1} T_w \quad \hbox{with $q_w=q_{s_1} \cdots q_{s_p}$ if $w=s_1 \cdots s_p$ is minimal length}$$ gives $T_s \epsilon=\epsilon$ for all $s \in S$ and hence $T_w \epsilon=\epsilon$ for all $w \in W$. It follows that the principal Schur element is
$$s=\sum_{w \in W} q_w^{-1}=q_{w_0}^{-1} \sum_{w \in W} q_w.$$ Combining this formula with Theorem \ref{support theorem1} recovers Etingof's result \cite{Eti}.

\subsection{The infinite family $G(r,1,n)$} \label{gr1n} Up to conjugacy, the generators of the monodromy for the braid group of $G(r,1,n)$ may be listed as $T_0,T_1,\dots,T_n$, and fixing parameters $\q,Q_0,Q_1,\dots,Q_{r-1}$ the Hecke relations are then
$$\prod_{0 \leq j \leq r-1} (T_0-Q_j)=0 \quad \text{and} \quad (T_i+1)(T_i-\q)=0 \quad \hbox{for $1 \leq i \leq n$.}$$ By Theorem 3.2 of \cite{ChJa} up to an invertible factor the principal Schur element of the group $G(r,1,n)$ is
$$s=[n]! \prod_{j=1}^{r-1} \prod_{m=0}^{n-1} (\q^m Q_0 Q_j^{-1}-1),$$ where
$$[n]!=(1+\q) (1+\q+\q^2) \cdots (1+\q+\cdots+\q^{n-1})$$ is the usual $\q$-factorial. The maximal parabolic subgroups of $G(r,1,n)$ are, up to conjugacy, of the form $S_k \times G(r,1,n-k)$, for $1 \leq k \leq n$. The quotient of the corresponding Schur elements is the $\q$-index, which up to an invertible factor is
$$[G(r,1,n):S_k \times G(r,1,n-k)]_\q=s_1/s_2=\left[ \begin{matrix} n \\ k \end{matrix} \right] \prod_{j=1}^{r-1} \prod_{m=n-k}^{n-1} (\q^m Q_0 Q_j^{-1}-1),$$ where
$$ \left[ \begin{matrix} n \\ k \end{matrix} \right] =\frac{[n]!}{[k]! [n-k]!}$$ is the usual $\q$-binomial coefficient.

By using Theorem \ref{support theorem2}, and writing the parameters of \cite{ChJa} with the notation of \cite{GGJL}, subsection 5.1, so that
$$q=e^{-2 \pi i c_0} \quad \text{and} \quad Q_j=e^{2 \pi i (j-d_j)/r},$$ we find that $L_c(\triv)$ is finite dimensional if and only if either
\begin{itemize}
\item[(a)] there exist integers $1 \leq j \leq r-1$ and $k>0$ with $k$ congruent to $-j$ mod $r$ and
$$d_0-d_j+r(n-1) c_0=k,$$ or 
\item[(b)] $c=\ell/d$ for some positive divisor $d$ of $n$ and some positive integer $\ell$ coprime to $d$, and also
$$d_0-d_j+rmc_0=k$$ for $n-d \leq m \leq n-1$ and integers $1 \leq j \leq r-1$ and $k>0$ with $k$ congruent to $-j$ mod $r$.
\end{itemize} Thus for the infinite family one might guess that when it is finite dimensional the irreducible quotient $L_c(\triv)$ may be constructed explicitly as described in subsection 5.2 of \cite{GGJL}.

\subsection{Exceptional groups}

Here we give data, computed using the development version \cite{Mi} of the CHEVIE package of GAP3 \cite{Sch}, for exceptional groups in the form of tables listing the principal Schur elements (from which one can read off the set of singular parameters), and the $q$-indices of each conjugacy class of maximal parabolic subgroups. Both are given up to an invertible factor. Note that one can find the list of conjugacy classes parabolic subgroups of primitive complex reflection groups in Appendix C of \cite{OrTe}. To save space, we set $G_{de,e,n} := G(de,e,n)$ and $H_2 := I_2(5) = G_{5,5,2}$. Moreover, we denote by $Z_n$ the cyclic group $G_{n, 1, 1}$.

The parameters for the Hecke algebra are the $q_{H,\chi}$. If $H_1$, $H_2$, \dots, are representative of the orbits of reflection hyperplanes
(in the order given by the function \texttt{HyperplaneOrbits}), we write $x_j$, $1 \leq j \leq e_{H_1} - 1$ for the indeterminates $q_{H_1, \det_{H_1}^j}$, then $y_j$ for those associated to $H_2$, etc.

The Schur elements are available in CHEVIE. They were computed for exceptional Coxeter groups in \cite{Sur,Ben,Lus1,Lus2,AlLu}, and for exceptional complex reflection groups in \cite{Mal1,Mal2} (under some assumptions allowing us to apply the results of \ref{symmetry}), see those papers and also \cite{MaMi}). Chlouveraki noticed that, up to a scalar and a monomial, they can be written as a product of cyclotomic polynomials evaluated at primitive monomials (that is, whose exponents are coprime), and this factorization is unique up to inversion of the monomials \cite{Chl}. This provides a compact way to display them. We write $\zeta_n$ for $e^{2\pi i / n}$, and $\Phi_n$ for the $n$th cyclotomic polynomial over $\QQ$. Moreover, we need cyclotomic polynomials over some cyclotomic extensions of $\QQ$, with the following notation (following GAP3):

\[
\begin{array}{c|c}
\hline
\Phi'_n(q),\ n \in\{3,4,6\} & q - \zeta_n\\
\Phi''_n(q),\ n \in\{3,4,6\} & q - \zeta_n^{-1}\\
\hline
\Phi_{12}' (q)& ( q - \zeta_{12} )( q - \zeta_{12}^5 )\\
\Phi_{12}''(q) & ( q - \zeta_{12}^7 )( q - \zeta_{12}^{11} )\\
\Phi_{12}'''(q) & ( q - \zeta_{12} )( q - \zeta_{12}^7 )\\
\Phi_{12}''''(q) & ( q - \zeta_{12}^5 )( q - \zeta_{12}^{11} )\\
\hline
\Phi_{30}'(q) & ( q - \zeta_{30} )( q - \zeta_{30}^{11} )( q - \zeta_{30}^{19} )( q - \zeta_{30}^{29} )\\
\Phi_{30}''(q) & ( q - \zeta_{30}^7)( q - \zeta_{30}^{13} )( q - \zeta_{30}^{17} )( q - \zeta_{30}^{23} )\\
\Phi_{30}'''(q) & ( q - \zeta_{30} )( q - \zeta_{30}^7 )( q - \zeta_{30}^{13} )( q - \zeta_{30}^{19} )\\
\Phi_{30}''''(q) & ( q - \zeta_{30}^{11} )( q - \zeta_{30}^{17} )( q - \zeta_{30}^{23} )( q - \zeta_{30}^{29} )\\
\hline
\end{array}
\]

\bigskip

\begin{landscape}


\[
\begin{array}{|c|l|}
\hline
|G_4|_q & \Phi_2{\Phi''_3}{\Phi''_6}(x_1)\Phi_2{\Phi'_3}{\Phi'_6}(x_2)\Phi_2(x_1x_2)\\
\hline
|G_4 : Z_3|_q & \Phi_2{\Phi''_6}(x_1)\Phi_2{\Phi'_6}(x_2)\Phi_2(x_1x_2)\\
\hline
\hline
|G_5|_q & {\Phi''_3}(x_1){\Phi'_3}(x_2){\Phi''_3}(y_1){\Phi'_3}(y_2){\Phi''_6}(x_1y_1)\Phi_2(x_1y_2)\Phi_2(x_2y_1){\Phi'_6}(x_2y_2)\Phi_2(x_1x_2y_1y_2)\\
\hline
|G_5 : Z_3'|_q
  & {\Phi''_3}(y_1){\Phi'_3}(y_2){\Phi''_6}(x_1y_1)\Phi_2(x_1y_2)\Phi_2(x_2y_1){\Phi'_6}(x_2y_2)\Phi_2(x_1x_2y_1y_2)\\
\hline
|G_5 : Z_3''|_q
  & {\Phi''_3}(x_1){\Phi'_3}(x_2){\Phi''_6}(x_1y_1)\Phi_2(x_1y_2)\Phi_2(x_2y_1){\Phi'_6}(x_2y_2)\Phi_2(x_1x_2y_1y_2)\\
\hline
\hline
|G_6|_q & \Phi_2(x_1){\Phi''_3}(y_1){\Phi'_3}(y_2)\Phi_2{\Phi''_6}(x_1y_1)\Phi_2{\Phi'_6}(x_1y_2)\Phi_2(x_1y_1y_2)\\
\hline
|G_6 : A_1|_q
& {\Phi''_3}(y_1){\Phi'_3}(y_2)\Phi_2{\Phi''_6}(x_1y_1)\Phi_2{\Phi'_6}(x_1y_2)\Phi_2(x_1y_1y_2)\\
\hline
|G_6 : Z_3|_q
  & \Phi_2(x_1)\Phi_2{\Phi''_6}(x_1y_1)\Phi_2{\Phi'_6}(x_1y_2)\Phi_2(x_1y_1y_2)\\
\hline
\hline
|G_7|_q 
& \Phi_2(x_1){\Phi''_3}(y_1){\Phi'_3}(y_2){\Phi''_3}(z_1){\Phi'_3}(z_2){\Phi''_6}(x_1y_1z_1)\Phi_2(x_1y_1z_2)\Phi_2(x_1y_2z_1){\Phi'_6}(x_1y_2z_2)\Phi_2(x_1y_1y_2z_1z_2)\\
\hline
|G_7 : A_1|_q
& {\Phi''_3}(y_1){\Phi'_3}(y_2){\Phi''_3}(z_1){\Phi'_3}(z_2){\Phi''_6}(x_1y_1z_1)\Phi_2(x_1y_1z_2)\Phi_2(x_1y_2z_1){\Phi'_6}(x_1y_2z_2)\Phi_2(x_1y_1y_2z_1z_2)\\
\hline
|G_7 : Z_3'|_q
& \Phi_2(x_1){\Phi''_3}(z_1){\Phi'_3}(z_2){\Phi''_6}(x_1y_1z_1)\Phi_2(x_1y_1z_2)\Phi_2(x_1y_2z_1){\Phi'_6}(x_1y_2z_2)\Phi_2(x_1y_1y_2z_1z_2)\\
\hline
|G_7 : Z_3''|_q
& \Phi_2(x_1){\Phi''_3}(y_1){\Phi'_3}(y_2){\Phi''_6}(x_1y_1z_1)\Phi_2(x_1y_1z_2)\Phi_2(x_1y_2z_1){\Phi'_6}(x_1y_2z_2)\Phi_2(x_1y_1y_2z_1z_2)\\
\hline
\hline
|G_8|_q
& {\Phi''_4}{\Phi''_{12}}(x_1)\Phi_2\Phi_3(x_2){\Phi'_4}{\Phi'_{12}}(x_3){\Phi''_4}(x_1x_2)\Phi_2(x_1x_3){\Phi'_4}(x_2x_3)\Phi_2(x_1x_2x_3)\\
\hline
|G_8 : Z_4|_q
& {\Phi''_{12}}(x_1)\Phi_3(x_2){\Phi'_{12}}(x_3){\Phi''_4}(x_1x_2)\Phi_2(x_1x_3){\Phi'_4}(x_2x_3)\Phi_2(x_1x_2x_3)\\
\hline
\hline
|G_9|_q
& \Phi_2(x_1){\Phi''_4}(y_1)\Phi_2(y_2){\Phi'_4}(y_3){\Phi''_{12}}(x_1y_1)\Phi_3(x_1y_2){\Phi'_{12}}(x_1y_3){\Phi''_4}(x_1y_1y_2)\Phi_2(x_1y_1y_3){\Phi'_4}(x_1y_2y_3)\Phi_2(x_1^2y_1y_2y_3)\\
\hline
|G_9 : A_1|_q
&  {\Phi''_4}(y_1)\Phi_2(y_2){\Phi'_4}(y_3){\Phi''_{12}}(x_1y_1)\Phi_3(x_1y_2){\Phi'_{12}}(x_1y_3){\Phi''_4}(x_1y_1y_2)\Phi_2(x_1y_1y_3){\Phi'_4}(x_1y_2y_3)\Phi_2(x_1^2y_1y_2y_3)\\
\hline
|G_9 : Z_4|_q
& \Phi_2(x_1){\Phi''_{12}}(x_1y_1)\Phi_3(x_1y_2){\Phi'_{12}}(x_1y_3){\Phi''_4}(x_1y_1y_2)\Phi_2(x_1y_1y_3){\Phi'_4}(x_1y_2y_3)\Phi_2(x_1^2y_1y_2y_3)\\
\hline
\hline
|G_{10}|_q
&\parbox{20cm}{\raggedright${
\Phi''_3}(x_1)\allowbreak
{\Phi'_3}(x_2)\allowbreak
{\Phi''_4}(y_1)\allowbreak
\Phi_2(y_2)\allowbreak
{\Phi'_4}(y_3)\allowbreak
(x_1y_1-\zeta_{12}^{11})\allowbreak
{\Phi''_3}(x_1y_2)\allowbreak
(x_1y_3-\zeta_{12}^5)\allowbreak
(x_2y_1\zeta_{12}^7)\allowbreak
{\Phi'_3}(x_2y_2)\allowbreak
(x_2y_3\zeta_{12})\allowbreak
{\Phi''_4}(x_1x_2y_1y_2)\allowbreak
\Phi_2(x_1x_2y_1y_3)\allowbreak
{\Phi'_4}(x_1x_2y_2y_3)\allowbreak
\Phi_2(x_1x_2y_1y_2y_3)$}\\
\hline
|G_{10} : Z_3|_q
&\text{\parbox{20cm}{\raggedright$
{\Phi''_3}(y_1)\allowbreak
{\Phi'_3}(y_2)\allowbreak
{\Phi''_4}(z_1)\allowbreak
\Phi_2(z_2)\allowbreak
{\Phi'_4}(z_3)\allowbreak
(x_1y_1z_1-\zeta_{12}^{11})\allowbreak
{\Phi''_3}(x_1y_1z_2)\allowbreak
(x_1y_1z_3-\zeta_{12}^5)\allowbreak
(x_1y_2z_1-\zeta_{12}^7)\allowbreak
{\Phi'_3}(x_1y_2z_2)\allowbreak
(x_1y_2z_3-\zeta_{12})\allowbreak
{\Phi''_4}(x_1y_1y_2z_1z_2)\allowbreak
\Phi_2(x_1y_1y_2z_1z_3)\allowbreak
{\Phi'_4}(x_1y_1y_2z_2z_3)
\allowbreak\Phi_2(x_1^2y_1y_2z_1z_2z_3)
$}}
\\
\hline
|G_{10} : Z_4|_q
&\text{\parbox{20cm}{\raggedright$\Phi_2(x_1)\allowbreak
{\Phi''_4}(z_1)\allowbreak
\Phi_2(z_2)\allowbreak
{\Phi'_4}(z_3)\allowbreak
(x_1y_1z_1-\zeta_{12}^{11})\allowbreak
{\Phi''_3}(x_1y_1z_2)\allowbreak
(x_1y_1z_3-\zeta_{12}^5)\allowbreak
(x_1y_2z_1-\zeta_{12}^7)\allowbreak
{\Phi'_3}(x_1y_2z_2)\allowbreak
(x_1y_2z_3-\zeta_{12})\allowbreak
{\Phi''_4}(x_1y_1y_2z_1z_2)\allowbreak
\Phi_2(x_1y_1y_2z_1z_3)\allowbreak
{\Phi'_4}(x_1y_1y_2z_2z_3)\allowbreak
\Phi_2(x_1^2y_1y_2z_1z_2z_3)$}}
\\
\hline
\end{array}
\]

\[
\begin{array}{|c|l|}
\hline
|G_{11}|_q
&\parbox{20cm}{\raggedright$
\Phi_2(x_1)\allowbreak
{\Phi''_3}(y_1)\allowbreak
{\Phi'_3}(y_2)\allowbreak
{\Phi''_4}(z_1)\allowbreak
\Phi_2(z_2)\allowbreak
{\Phi'_4}(z_3)\allowbreak
(x_1y_1z_1-\zeta_{12}^{11})\allowbreak
{\Phi''_3}(x_1y_1z_2)\allowbreak
(x_1y_1z_3-\zeta_{12}^5)\allowbreak
(x_1y_2z_1-\zeta_{12}^7)\allowbreak
{\Phi'_3}(x_1y_2z_2)\allowbreak
(x_1y_2z_3-\zeta_{12})\allowbreak
{\Phi''_4}(x_1y_1y_2z_1z_2)\allowbreak
\Phi_2(x_1y_1y_2z_1z_3)\allowbreak
{\Phi'_4}(x_1y_1y_2z_2z_3)\allowbreak
\Phi_2(x_1^2y_1y_2z_1z_2z_3)
$}\\
\hline
|G_{11} : A_1|_q
& \text{ \parbox{20cm}{\raggedright$
{\Phi''_3}(y_1)\allowbreak
{\Phi'_3}(y_2)\allowbreak
{\Phi''_4}(z_1)\allowbreak
\Phi_2(z_2)\allowbreak
{\Phi'_4}(z_3)\allowbreak
(x_1y_1z_1-\zeta_{12}^{11})\allowbreak
{\Phi''_3}(x_1y_1z_2)\allowbreak
(x_1y_1z_3-\zeta_{12}^5) \allowbreak
(x_1y_2z_1-\zeta_{12}^7) \allowbreak
{\Phi'_3}(x_1y_2z_2) \allowbreak
(x_1y_2z_3-\zeta_{12})\allowbreak
{\Phi''_4}(x_1y_1y_2z_1z_2)\allowbreak
\Phi_2(x_1y_1y_2z_1z_3)\allowbreak
{\Phi'_4}(x_1y_1y_2z_2z_3)\allowbreak
\Phi_2(x_1^2y_1y_2z_1z_2z_3)
$}}
\\
\hline
|G_{11} : Z_3|_q
&\text{\parbox{20cm}{\raggedright$
\Phi_2(x_1)\allowbreak
{\Phi''_4}(z_1)\allowbreak
\Phi_2(z_2)\allowbreak
{\Phi'_4}(z_3)\allowbreak
(x_1y_1z_1-\zeta_{12}^{11})\allowbreak
{\Phi''_3}(x_1y_1z_2)\allowbreak
(x_1y_1z_3-\zeta_{12}^5)\allowbreak
(x_1y_2z_1-\zeta_{12}^7)\allowbreak
{\Phi'_3}(x_1y_2z_2)\allowbreak
(x_1y_2z_3-\zeta_{12})\allowbreak
{\Phi''_4}(x_1y_1y_2z_1z_2)\allowbreak
\Phi_2(x_1y_1y_2z_1z_3)\allowbreak
{\Phi'_4}(x_1y_1y_2z_2z_3)\allowbreak
\Phi_2(x_1^2y_1y_2z_1z_2z_3)$}}
\\
\hline
|G_{11} : Z_4|_q
&\text{\parbox{20cm}{\raggedright$
\Phi_2(x_1)\allowbreak
{\Phi''_3}(y_1)\allowbreak
{\Phi'_3}(y_2)\allowbreak
(x_1y_1z_1-\zeta_{12}^{11})\allowbreak
{\Phi''_3}(x_1y_1z_2)\allowbreak
(x_1y_1z_3-\zeta_{12}^5)\allowbreak
(x_1y_2z_1-\zeta_{12}^7)\allowbreak
{\Phi'_3}(x_1y_2z_2)\allowbreak
(x_1y_2z_3-\zeta_{12})\allowbreak
{\Phi''_4}(x_1y_1y_2z_1z_2)\allowbreak
\Phi_2(x_1y_1y_2z_1z_3)\allowbreak
{\Phi'_4}(x_1y_1y_2z_2z_3)\allowbreak
\Phi_2(x_1^2y_1y_2z_1z_2z_3)$}}
\\
\hline
\hline
|G_{12}|_q & \Phi_2^2\Phi_3\Phi_4^2\Phi_{12}(x_1)\\
\hline
|G_{12} : A_1|_q
& \Phi_2\Phi_3\Phi_4^2\Phi_{12}(x_1)\\
\hline
\hline
|G_{13}|_q & \Phi_2(x_1)\Phi_2\Phi_3(y_1)\Phi_2(x_1y_1)\Phi_2^2\Phi_6(x_1y_1^2)\\
\hline
|G_{13} : A_1'|_q
& \Phi_2(x_1)\Phi_3(y_1)\Phi_2(x_1y_1)\Phi_2^2\Phi_6(x_1y_1^2)\\
\hline
|G_{13} : A_1''|_q
& \Phi_2\Phi_3(y_1)\Phi_2(x_1y_1)\Phi_2^2\Phi_6(x_1y_1^2)\\
\hline
\hline
|G_{14}|_q
&\Phi_2(x_1){\Phi''_3}(y_1){\Phi'_3}(y_2){\Phi''_3}{\Phi''''_{12}}(x_1y_1){\Phi'_3}{\Phi'''_{12}}(x_1y_2)\Phi_2\Phi_4(x_1y_1y_2)\Phi_2(x_1^2y_1y_2)\\
\hline
|G_{14} : A_1|_q
& {\Phi''_3}(y_1){\Phi'_3}(y_2){\Phi''_3}{\Phi''''_{12}}(x_1y_1){\Phi'_3}{\Phi'''_{12}}(x_1y_2)\Phi_2\Phi_4(x_1y_1y_2)\Phi_2(x_1^2y_1y_2)
\\
\hline
|G_{14} : Z_3|_q
& \Phi_2(x_1){\Phi''_3}{\Phi''''_{12}}(x_1y_1){\Phi'_3}{\Phi'''_{12}}(x_1y_2)\Phi_2\Phi_4(x_1y_1y_2)\Phi_2(x_1^2y_1y_2)\\
\hline
\hline
|G_{15}|_q
& \Phi_2(x_1){\Phi''_3}(y_1){\Phi'_3}(y_2)\Phi_2(z_1){\Phi''_3}(x_1y_1){\Phi'_3}(x_1y_2)\Phi_2(x_1y_1y_2z_1)\Phi_2(x_1^2y_1y_2z_1){\Phi''_6}(x_1^2y_1^2z_1){\Phi'_6}(x_1^2y_2^2z_1)\Phi_2(x_1^2y_1^2y_2^2z_1)\\
\hline
|G_{15} : A_1'|_q
  & {\Phi''_3}(y_1){\Phi'_3}(y_2)\Phi_2(z_1){\Phi''_3}(x_1y_1){\Phi'_3}(x_1y_2)\Phi_2(x_1y_1y_2z_1)\Phi_2(x_1^2y_1y_2z_1){\Phi''_6}(x_1^2y_1^2z_1){\Phi'_6}(x_1^2y_2^2z_1)\Phi_2(x_1^2y_1^2y_2^2z_1)\\
\hline
|G_{15} : Z_3|_q
  & \Phi_2(x_1)\Phi_2(z_1){\Phi''_3}(x_1y_1){\Phi'_3}(x_1y_2)\Phi_2(x_1y_1y_2z_1)\Phi_2(x_1^2y_1y_2z_1){\Phi''_6}(x_1^2y_1^2z_1){\Phi'_6}(x_1^2y_2^2z_1)\Phi_2(x_1^2y_1^2y_2^2z_1)
\\
\hline
|G_{15} : A_1''|_q
& \Phi_2(x_1){\Phi''_3}(y_1){\Phi'_3}(y_2){\Phi''_3}(x_1y_1){\Phi'_3}(x_1y_2)\Phi_2(x_1y_1y_2z_1)\Phi_2(x_1^2y_1y_2z_1){\Phi''_6}(x_1^2y_1^2z_1){\Phi'_6}(x_1^2y_2^2z_1)\Phi_2(x_1^2y_1^2y_2^2z_1)\\
\hline
\hline
|G_{16}|_q 
&\parbox{20cm}{$
(x_1-\zeta_5^4)(x_1+\zeta_{15}^2)(x_1+\zeta_{15}^7)
(x_2-\zeta_5^3)(x_2+\zeta_{15}^{14})(x_2+\zeta_{15}^4)
(x_3-\zeta_5^2)(x_3+\zeta_{15}^{11})(x_3+\zeta_{15})
(x_4-\zeta_5)(x_4+\zeta_{15}^8)(x_4+\zeta_{15}^{13})
(x_1x_2+\zeta_5^2)
(x_1x_3+\zeta_5)\allowbreak
\Phi_2(x_1x_4)\allowbreak
\Phi_2(x_2x_3)\allowbreak
(x_2x_4+\zeta_5^4)\allowbreak
(x_3x_4+\zeta_5^3)\allowbreak
(x_1x_2x_3-\zeta_5^4)\allowbreak
(x_1x_2x_4-\zeta_5^3)\allowbreak
(x_1x_3x_4-\zeta_5^2)\allowbreak
(x_2x_3x_4-\zeta_5)\allowbreak
\Phi_2\Phi_3(x_1x_2x_3x_4)$}\\
\hline
|G_{16}:Z_5|_q
&\text{ \parbox{20cm}{\raggedright$
(x_1+\zeta_{15}^2)(x_1+\zeta_{15}^7)\allowbreak
(x_2+\zeta_{15}^{14})(x_2+\zeta_{15}^4)\allowbreak
(x_3+\zeta_{15}^{11})(x_3+\zeta_{15})\allowbreak
(x_4+\zeta_{15}^8)(x_4+\zeta_{15}^{13})\allowbreak
(x_1x_2+\zeta_5^2)\allowbreak
(x_1x_3+\zeta_5)\allowbreak
(x_2x_4+\zeta_5^4)\allowbreak
(x_3x_4+\zeta_5^3)\allowbreak
(x_1x_2x_3-\zeta_5^4)\allowbreak
(x_1x_2x_4-\zeta_5^3)\allowbreak
(x_1x_3x_4-\zeta_5^2)\allowbreak
(x_2x_3x_4-\zeta_5)\allowbreak
\Phi_2(x_1x_4)\allowbreak
\Phi_2(x_2x_3)\allowbreak
\Phi_2\Phi_3(x_1x_2x_3x_4)
$}}
\\
\hline
\hline
|G_{17}|_q
& \parbox{20cm}{$
\Phi_2(x_1)\allowbreak
(y_1-\zeta_5^4)\allowbreak
(y_2-\zeta_5^3)\allowbreak
(y_3-\zeta_5^2)\allowbreak
(y_4-\zeta_5)\allowbreak
(x_1y_1+\zeta_{15}^2)(x_1y_1+\zeta_{15}^7)\allowbreak
(x_1y_2+\zeta_{15}^{14})(x_1y_2+\zeta_{15}^4)\allowbreak
(x_1y_3+\zeta_{15}^{11})(x_1y_3+\zeta_{15})\allowbreak
(x_1y_4+\zeta_{15}^8)(x_1y_4+\zeta_{15}^{13})\allowbreak
(x_1y_1y_2+\zeta_5^2)\allowbreak
(x_1y_1y_3+\zeta_5)\allowbreak
\Phi_2(x_1y_1y_4)\allowbreak
\Phi_2(x_1y_2y_3)\allowbreak
(x_1y_2y_4+\zeta_5^4)\allowbreak
(x_1y_3y_4+\zeta_5^3)\allowbreak
(x_1^2y_1y_2y_3-\zeta_5^4)\allowbreak
(x_1^2y_1y_2y_4-\zeta_5^3)\allowbreak
(x_1^2y_1y_3y_4-\zeta_5^2)\allowbreak
(x_1^2y_2y_3y_4-\zeta_5)\allowbreak
\Phi_3(x_1^2y_1y_2y_3y_4)\allowbreak
\Phi_2(x_1^3y_1y_2y_3y_4)\allowbreak
$}\\
\hline
|G_{17}:A_1|_q
&\text{ \parbox{20cm}{\raggedright$
(y_1-\zeta_5^4)\allowbreak
(y_2-\zeta_5^3)\allowbreak
(y_3-\zeta_5^2)\allowbreak
(y_4-\zeta_5)\allowbreak
(x_1y_1+\zeta_{15}^2)(x_1y_1+\zeta_{15}^7)\allowbreak
(x_1y_2+\zeta_{15}^{14})(x_1y_2+\zeta_{15}^4)\allowbreak
(x_1y_3+\zeta_{15}^{11})(x_1y_3+\zeta_{15})\allowbreak
(x_1y_4+\zeta_{15}^8)(x_1y_4+\zeta_{15}^{13})\allowbreak
(x_1y_1y_2+\zeta_5^2)\allowbreak
(x_1y_1y_3+\zeta_5)\allowbreak
\Phi_2(x_1y_1y_4)\allowbreak
\Phi_2(x_1y_2y_3)\allowbreak
(x_1y_2y_4+\zeta_5^4)\allowbreak
(x_1y_3y_4+\zeta_5^3)\allowbreak
(x_1^2y_1y_2y_3-\zeta_5^4)\allowbreak
(x_1^2y_1y_2y_4-\zeta_5^3)\allowbreak
(x_1^2y_1y_3y_4-\zeta_5^2)\allowbreak
(x_1^2y_2y_3y_4-\zeta_5)\allowbreak
\Phi_3(x_1^2y_1y_2y_3y_4)\allowbreak
\Phi_2(x_1^3y_1y_2y_3y_4)
$}}
\\
\hline
|G_{17}:Z_5|_q
&
\text{ \parbox{20cm}{\raggedright$
\Phi_2(x_1)\allowbreak
(x_1y_1+\zeta_{15}^2)\allowbreak
(x_1y_1+\zeta_{15}^7)\allowbreak
(x_1y_2+\zeta_{15}^{14})\allowbreak
(x_1y_2+\zeta_{15}^4)\allowbreak
(x_1y_3+\zeta_{15}^{11})\allowbreak
(x_1y_3+\zeta_{15})\allowbreak
(x_1y_4+\zeta_{15}^8)\allowbreak
(x_1y_4+\zeta_{15}^{13})\allowbreak
(x_1y_1y_2+\zeta_5^2)\allowbreak
(x_1y_1y_3+\zeta_5)\allowbreak
\Phi_2(x_1y_1y_4)\allowbreak
\Phi_2(x_1y_2y_3)\allowbreak
(x_1y_2y_4+\zeta_5^4)\allowbreak
(x_1y_3y_4+\zeta_5^3)\allowbreak
(x_1^2y_1y_2y_3-\zeta_5^4)\allowbreak
(x_1^2y_1y_2y_4-\zeta_5^3)\allowbreak
(x_1^2y_1y_3y_4-\zeta_5^2)\allowbreak
(x_1^2y_2y_3y_4-\zeta_5)\allowbreak
\Phi_3(x_1^2y_1y_2y_3y_4)\allowbreak
\Phi_2(x_1^3y_1y_2y_3y_4)
$}}
\\
\hline
\end{array}
\]

\[
\begin{array}{|c|l|}
\hline
|G_{18}|_q
& \parbox{20cm}{$
{\Phi''_3}(x_1)\allowbreak
{\Phi'_3}(x_2)\allowbreak
(y_1-\zeta_5^4)\allowbreak
(y_2-\zeta_5^3)\allowbreak
(y_3-\zeta_5^2)\allowbreak
(y_4-\zeta_5)\allowbreak
(x_1y_1+\zeta_{15}^7)\allowbreak
(x_1y_2+\zeta_{15}^4)\allowbreak
(x_1y_3+\zeta_{15})\allowbreak
(x_1y_4+\zeta_{15}^{13})\allowbreak
(x_2y_1+\zeta_{15}^2)\allowbreak
(x_2y_2+\zeta_{15}^{14})\allowbreak
(x_2y_3+\zeta_{15}^{11})\allowbreak
(x_2y_4+\zeta_{15}^8)\allowbreak
(x_1x_2y_1y_2+\zeta_5^2)\allowbreak
(x_1x_2y_1y_3+\zeta_5)\allowbreak
\Phi_2(x_1x_2y_1y_4)\allowbreak
\Phi_2(x_1x_2y_2y_3)\allowbreak
(x_1x_2y_2y_4+\zeta_5^4)\allowbreak
(x_1x_2y_3y_4+\zeta_5^3)\allowbreak
(x_1x_2y_1y_2y_3-\zeta_5^4)\allowbreak
(x_1x_2y_1y_2y_4-\zeta_5^3)\allowbreak
(x_1x_2y_1y_3y_4-\zeta_5^2)\allowbreak
(x_1x_2y_2y_3y_4-\zeta_5)\allowbreak
{\Phi'_3}(x_1x_2^2y_1y_2y_3y_4)\allowbreak
{\Phi''_3}(x_1^2x_2y_1y_2y_3y_4)\allowbreak
\Phi_2(x_1^2x_2^2y_1y_2y_3y_4)\allowbreak
$}\\
\hline
|G_{18}:Z_3|_q
&
\text{ \parbox{20cm}{\raggedright$
(y_1-\zeta_5^4)\allowbreak
(y_2-\zeta_5^3)\allowbreak
(y_3-\zeta_5^2)\allowbreak
(y_4-\zeta_5)\allowbreak
(x_1y_1+\zeta_{15}^7)\allowbreak
(x_1y_2+\zeta_{15}^4)\allowbreak
(x_1y_3+\zeta_{15})\allowbreak
(x_1y_4+\zeta_{15}^{13})\allowbreak
(x_2y_1+\zeta_{15}^2)\allowbreak
(x_2y_2+\zeta_{15}^{14})\allowbreak
(x_2y_3+\zeta_{15}^{11})\allowbreak
(x_2y_4+\zeta_{15}^8)\allowbreak
(x_1x_2y_1y_2+\zeta_5^2)\allowbreak
(x_1x_2y_1y_3+\zeta_5)\allowbreak
\Phi_2(x_1x_2y_1y_4)\allowbreak
\Phi_2(x_1x_2y_2y_3)\allowbreak
(x_1x_2y_2y_4+\zeta_5^4)\allowbreak
(x_1x_2y_3y_4+\zeta_5^3)\allowbreak
(x_1x_2y_1y_2y_3-\zeta_5^4)\allowbreak
(x_1x_2y_1y_2y_4-\zeta_5^3)\allowbreak
(x_1x_2y_1y_3y_4-\zeta_5^2)\allowbreak
(x_1x_2y_2y_3y_4-\zeta_5)\allowbreak
{\Phi'_3}(x_1x_2^2y_1y_2y_3y_4)\allowbreak
{\Phi''_3}(x_1^2x_2y_1y_2y_3y_4)\allowbreak
\Phi_2(x_1^2x_2^2y_1y_2y_3y_4)
$}}
\\
\hline
|G_{18}:Z_5|_q
&
\text{ \parbox{20cm}{\raggedright$
{\Phi''_3}(x_1)\allowbreak
{\Phi'_3}(x_2)\allowbreak
(x_1y_1+\zeta_{15}^7)\allowbreak
(x_1y_2+\zeta_{15}^4)\allowbreak
(x_1y_3+\zeta_{15})\allowbreak
(x_1y_4+\zeta_{15}^{13})\allowbreak
(x_2y_1+\zeta_{15}^2)\allowbreak
(x_2y_2+\zeta_{15}^{14})\allowbreak
(x_2y_3+\zeta_{15}^{11})\allowbreak
(x_2y_4+\zeta_{15}^8)\allowbreak
(x_1x_2y_1y_2+\zeta_5^2)\allowbreak
(x_1x_2y_1y_3+\zeta_5)\allowbreak
\Phi_2(x_1x_2y_1y_4)\allowbreak
\Phi_2(x_1x_2y_2y_3)\allowbreak
(x_1x_2y_2y_4+\zeta_5^4)\allowbreak
(x_1x_2y_3y_4+\zeta_5^3)\allowbreak
(x_1x_2y_1y_2y_3-\zeta_5^4)\allowbreak
(x_1x_2y_1y_2y_4-\zeta_5^3)\allowbreak
(x_1x_2y_1y_3y_4-\zeta_5^2)\allowbreak
(x_1x_2y_2y_3y_4-\zeta_5)\allowbreak
{\Phi'_3}(x_1x_2^2y_1y_2y_3y_4)\allowbreak
{\Phi''_3}(x_1^2x_2y_1y_2y_3y_4)\allowbreak
\Phi_2(x_1^2x_2^2y_1y_2y_3y_4)
$}}
\\
\hline
\hline
|G_{19}|_q
&\text{ \parbox{20cm}{\raggedright$
\Phi_2(x_1)\allowbreak
{\Phi''_3}(y_1)\allowbreak
{\Phi'_3}(y_2)\allowbreak
(z_1-\zeta_5^4)\allowbreak
(z_2-\zeta_5^3)\allowbreak
(z_3-\zeta_5^2)\allowbreak
(z_4-\zeta_5)\allowbreak
(x_1y_1z_1+\zeta_{15}^7)\allowbreak
(x_1y_1z_2+\zeta_{15}^4)\allowbreak
(x_1y_1z_3+\zeta_{15})\allowbreak
(x_1y_1z_4+\zeta_{15}^{13})\allowbreak
(x_1y_2z_1+\zeta_{15}^2)\allowbreak
(x_1y_2z_2+\zeta_{15}^{14})\allowbreak
(x_1y_2z_3+\zeta_{15}^{11})\allowbreak
(x_1y_2z_4+\zeta_{15}^8)\allowbreak
(x_1y_1y_2z_1z_2+\zeta_5^2)\allowbreak
(x_1y_1y_2z_1z_3+\zeta_5)\allowbreak
\Phi_2(x_1y_1y_2z_1z_4)\allowbreak
\Phi_2(x_1y_1y_2z_2z_3)\allowbreak
(x_1y_1y_2z_2z_4+\zeta_5^4)\allowbreak
(x_1y_1y_2z_3z_4+\zeta_5^3)\allowbreak
(x_1^2y_1y_2z_1z_2z_3-\zeta_5^4)\allowbreak
(x_1^2y_1y_2z_1z_2z_4-\zeta_5^3)\allowbreak
(x_1^2y_1y_2z_1z_3z_4-\zeta_5^2)\allowbreak
(x_1^2y_1y_2z_2z_3z_4-\zeta_5)\allowbreak
{\Phi'_3}(x_1^2y_1y_2^2z_1z_2z_3z_4)\allowbreak
{\Phi''_3}(x_1^2y_1^2y_2z_1z_2z_3z_4)\allowbreak
\Phi_2(x_1^3y_1^2y_2^2z_1z_2z_3z_4)
$}}\\
\hline
|G_{19}:A_1|_q
&\text{ \parbox{20cm}{\raggedright$
{\Phi''_3}(y_1)\allowbreak
{\Phi'_3}(y_2)\allowbreak
(z_1-\zeta_5^4)\allowbreak
(z_2-\zeta_5^3)\allowbreak
(z_3-\zeta_5^2)\allowbreak
(z_4-\zeta_5)\allowbreak
(x_1y_1z_1+\zeta_{15}^7)\allowbreak
(x_1y_1z_2+\zeta_{15}^4)\allowbreak
(x_1y_1z_3+\zeta_{15})\allowbreak
(x_1y_1z_4+\zeta_{15}^{13})\allowbreak
(x_1y_2z_1+\zeta_{15}^2)\allowbreak
(x_1y_2z_2+\zeta_{15}^{14})\allowbreak
(x_1y_2z_3+\zeta_{15}^{11})\allowbreak
(x_1y_2z_4+\zeta_{15}^8)\allowbreak
(x_1y_1y_2z_1z_2+\zeta_5^2)\allowbreak
(x_1y_1y_2z_1z_3+\zeta_5)\allowbreak
\Phi_2(x_1y_1y_2z_1z_4)\allowbreak
\Phi_2(x_1y_1y_2z_2z_3)\allowbreak
(x_1y_1y_2z_2z_4+\zeta_5^4)\allowbreak
(x_1y_1y_2z_3z_4+\zeta_5^3)\allowbreak
(x_1^2y_1y_2z_1z_2z_3-\zeta_5^4)\allowbreak
(x_1^2y_1y_2z_1z_2z_4-\zeta_5^3)\allowbreak
(x_1^2y_1y_2z_1z_3z_4-\zeta_5^2)\allowbreak
(x_1^2y_1y_2z_2z_3z_4-\zeta_5)\allowbreak
{\Phi'_3}(x_1^2y_1y_2^2z_1z_2z_3z_4)\allowbreak
{\Phi''_3}(x_1^2y_1^2y_2z_1z_2z_3z_4)\allowbreak
\Phi_2(x_1^3y_1^2y_2^2z_1z_2z_3z_4)
$}}
\\
\hline
|G_{19}:Z_3|_q
&
\text{ \parbox{20cm}{\raggedright$
\Phi_2(x_1)\allowbreak
(z_1-\zeta_5^4)\allowbreak
(z_2-\zeta_5^3)\allowbreak
(z_3-\zeta_5^2)\allowbreak
(z_4-\zeta_5)\allowbreak
(x_1y_1z_1+\zeta_{15}^7)\allowbreak
(x_1y_1z_2+\zeta_{15}^4)\allowbreak
(x_1y_1z_3+\zeta_{15})\allowbreak
(x_1y_1z_4+\zeta_{15}^{13})\allowbreak
(x_1y_2z_1+\zeta_{15}^2)\allowbreak
(x_1y_2z_2+\zeta_{15}^{14})\allowbreak
(x_1y_2z_3+\zeta_{15}^{11})\allowbreak
(x_1y_2z_4+\zeta_{15}^8)\allowbreak
(x_1y_1y_2z_1z_2+\zeta_5^2)\allowbreak
(x_1y_1y_2z_1z_3+\zeta_5)\allowbreak
\Phi_2(x_1y_1y_2z_1z_4)\allowbreak
\Phi_2(x_1y_1y_2z_2z_3)\allowbreak
(x_1y_1y_2z_2z_4+\zeta_5^4)\allowbreak
(x_1y_1y_2z_3z_4+\zeta_5^3)\allowbreak
(x_1^2y_1y_2z_1z_2z_3-\zeta_5^4)\allowbreak
(x_1^2y_1y_2z_1z_2z_4-\zeta_5^3)\allowbreak
(x_1^2y_1y_2z_1z_3z_4-\zeta_5^2)\allowbreak
(x_1^2y_1y_2z_2z_3z_4-\zeta_5)\allowbreak
{\Phi'_3}(x_1^2y_1y_2^2z_1z_2z_3z_4)\allowbreak
{\Phi''_3}(x_1^2y_1^2y_2z_1z_2z_3z_4)\allowbreak
\Phi_2(x_1^3y_1^2y_2^2z_1z_2z_3z_4)
$}}
\\
\hline
|G_{19}:Z_5|_q
&\text{ \parbox{20cm}{\raggedright$
\Phi_2(x_1)\allowbreak
{\Phi''_3}(y_1)\allowbreak
{\Phi'_3}(y_2)\allowbreak
(x_1y_1z_1+\zeta_{15}^7)\allowbreak
(x_1y_1z_2+\zeta_{15}^4)\allowbreak
(x_1y_1z_3+\zeta_{15})\allowbreak
(x_1y_1z_4+\zeta_{15}^{13})\allowbreak
(x_1y_2z_1+\zeta_{15}^2)\allowbreak
(x_1y_2z_2+\zeta_{15}^{14})\allowbreak
(x_1y_2z_3+\zeta_{15}^{11})\allowbreak
(x_1y_2z_4+\zeta_{15}^8)\allowbreak
(x_1y_1y_2z_1z_2+\zeta_5^2)\allowbreak
(x_1y_1y_2z_1z_3+\zeta_5)\allowbreak
\Phi_2(x_1y_1y_2z_1z_4)\allowbreak
\Phi_2(x_1y_1y_2z_2z_3)\allowbreak
(x_1y_1y_2z_2z_4+\zeta_5^4)\allowbreak
(x_1y_1y_2z_3z_4+\zeta_5^3)\allowbreak
(x_1^2y_1y_2z_1z_2z_3-\zeta_5^4)\allowbreak
(x_1^2y_1y_2z_1z_2z_4-\zeta_5^3)\allowbreak
(x_1^2y_1y_2z_1z_3z_4-\zeta_5^2)\allowbreak
(x_1^2y_1y_2z_2z_3z_4-\zeta_5)\allowbreak
{\Phi'_3}(x_1^2y_1y_2^2z_1z_2z_3z_4)\allowbreak
{\Phi''_3}(x_1^2y_1^2y_2z_1z_2z_3z_4)\allowbreak
\Phi_2(x_1^3y_1^2y_2^2z_1z_2z_3z_4)
$}}
\\
\hline
\end{array}
\]

\[
\begin{array}{|c|l|}
\hline
|G_{20}|_q
&{\Phi''_3}{\Phi''''_{30}}(x_1){\Phi'_3}{\Phi'''_{30}}(x_2)\Phi_2^2\Phi_4\Phi_5\Phi_{10}(x_1x_2){\Phi'_3}(x_1x_2^2){\Phi''_3}(x_1^2x_2)\\
\hline
|G_{20}:Z_3|_q
& {\Phi''''_{30}}(x_1){\Phi'''_{30}}(x_2)\Phi_2^2\Phi_4\Phi_5\Phi_{10}(x_1x_2){\Phi'_3}(x_1x_2^2){\Phi''_3}(x_1^2x_2)\\
\hline
\hline
|G_{21}|_q&
\Phi_2(x_1){\Phi''_3}(y_1){\Phi'_3}(y_2){\Phi''''_{30}}(x_1y_1){\Phi'''_{30}}(x_1y_2)\Phi_2^2\Phi_{10}(x_1y_1y_2)\Phi_5(x_1^2y_1y_2){\Phi'_3}(x_1^2y_1y_2^2){\Phi''_3}(x_1^2y_1^2y_2)\Phi_2(x_1^3y_1^2y_2^2)\\
\hline
|G_{21} : A_1|_q
& {\Phi''_3}(y_1){\Phi'_3}(y_2){\Phi''''_{30}}(x_1y_1){\Phi'''_{30}}(x_1y_2)\Phi_2^2\Phi_{10}(x_1y_1y_2)\Phi_5(x_1^2y_1y_2){\Phi'_3}(x_1^2y_1y_2^2){\Phi''_3}(x_1^2y_1^2y_2)\Phi_2(x_1^3y_1^2y_2^2)\\
\hline
|G_{21} : Z_3|_q
& \Phi_2(x_1){\Phi''''_{30}}(x_1y_1){\Phi'''_{30}}(x_1y_2)\Phi_2^2\Phi_{10}(x_1y_1y_2)\Phi_5(x_1^2y_1y_2){\Phi'_3}(x_1^2y_1y_2^2){\Phi''_3}(x_1^2y_1^2y_2)\Phi_2(x_1^3y_1^2y_2^2)\\
\hline
\hline
|G_{22}|_q
& \Phi_2^4\Phi_3\Phi_5\Phi_6^2\Phi_{10}^2\Phi_{30}(x_1)\\
\hline
|G_{22} : A_1|_q
& \Phi_2^3\Phi_3\Phi_5\Phi_6^2\Phi_{10}^2\Phi_{30}(x_1)\\
\hline
|G_{23}|_q = |H_3|_q&\Phi_2^3\Phi_3\Phi_5\Phi_6\Phi_{10}(x_1)\\
\hline
|G_{23} : H_2|_q & \Phi_2^2\Phi_3\Phi_6\Phi_{10}(x_1)\\
\hline
|G_{23} : A_1^2|_q & \Phi_2\Phi_3\Phi_5\Phi_6\Phi_{10}(x_1)\\
\hline
|G_{23} : A_2|_q & \Phi_2^2\Phi_5\Phi_6\Phi_{10}(x_1)\\
\hline
\hline
|G_{24}|_q&\Phi_2^3\Phi_3\Phi_4\Phi_6\Phi_7\Phi_{14}(x_1)\\
\hline
|G_{24} : A_2|_q
  & \Phi_2^2\Phi_4\Phi_6\Phi_7\Phi_{14}(x_1)\\
\hline
|G_{24} : B_2|_q
  & \Phi_2\Phi_3\Phi_6\Phi_7\Phi_{14}(x_1)\\
\hline
\hline
|G_{25}|_q&\Phi_2{\Phi''_3}^2{\Phi''_6}{\Phi''''_{12}}(x_1)\Phi_2{\Phi'_3}^2{\Phi'_6}{\Phi'''_{12}}(x_2)\Phi_2\Phi_3(x_1x_2){\Phi'_3}(x_1x_2^2){\Phi''_3}(x_1^2x_2)\\
\hline
|G_{25} : G_4|_q
  & {\Phi''_3}{\Phi''''_{12}}(x_1){\Phi'_3}{\Phi'''_{12}}(x_2)\Phi_3(x_1x_2){\Phi'_3}(x_1x_2^2){\Phi''_3}(x_1^2x_2)\\
\hline
|G_{25} : Z_3^2|_q
  & \Phi_2{\Phi''_6}{\Phi''''_{12}}(x_1)\Phi_2{\Phi'_6}{\Phi'''_{12}}(x_2)\Phi_2\Phi_3(x_1x_2){\Phi'_3}(x_1x_2^2){\Phi''_3}(x_1^2x_2)\\
\hline
\hline
|G_{26}|_q&\Phi_2(x_1)\Phi_2{\Phi''_3}{\Phi''_6}(y_1)\Phi_2{\Phi'_3}{\Phi'_6}(y_2){\Phi''_3}(x_1y_1){\Phi'_3}(x_1y_2)\Phi_2(y_1y_2)\Phi_3(x_1y_1y_2){\Phi'_3}(x_1y_1y_2^2){\Phi''_6}(x_1y_1^2){\Phi'_6}(x_1y_2^2){\Phi''_3}(x_1y_1^2y_2)\\
\hline
|G_{26} : G_{3,1,2}|_q
  & \Phi_2{\Phi''_6}(y_1)\Phi_2{\Phi'_6}(y_2)\Phi_2(y_1y_2)\Phi_3(x_1y_1y_2){\Phi'_3}(x_1y_1y_2^2){\Phi''_6}(x_1y_1^2){\Phi'_6}(x_1y_2^2){\Phi''_3}(x_1y_1^2y_2)\\
\hline
|G_{26} : A_1 Z_3|_q
  & \Phi_2{\Phi''_6}(y_1)\Phi_2{\Phi'_6}(y_2){\Phi''_3}(x_1y_1){\Phi'_3}(x_1y_2)\Phi_2(y_1y_2)\Phi_3(x_1y_1y_2){\Phi'_3}(x_1y_1y_2^2){\Phi''_6}(x_1y_1^2){\Phi'_6}(x_1y_2^2){\Phi''_3}(x_1y_1^2y_2)\\
\hline
|G_{26} : G_4|_q
  & \Phi_2(x_1){\Phi''_3}(x_1y_1){\Phi'_3}(x_1y_2)\Phi_3(x_1y_1y_2){\Phi'_3}(x_1y_1y_2^2){\Phi''_6}(x_1y_1^2){\Phi'_6}(x_1y_2^2){\Phi''_3}(x_1y_1^2y_2)\\
\hline
\hline
|G_{27}|_q&\Phi_2^3\Phi_3^3\Phi_4\Phi_5\Phi_6^3\Phi_{10}\Phi_{12}\Phi_{15}\Phi_{30}(x_1)\\
\hline
|G_{27} : A'_2|_q = |G_{27} : A''_2|_q & \Phi_2^2\Phi_3^2\Phi_4\Phi_5\Phi_6^3\Phi_{10}\Phi_{12}\Phi_{15}\Phi_{30}(x_1)\\
\hline
|G_{27} : B_2|_q & \Phi_2\Phi_3^3\Phi_5\Phi_6^3\Phi_{10}\Phi_{12}\Phi_{15}\Phi_{30}(x_1)\\
\hline
|G_{27} : H_2|_q & \Phi_2^2\Phi_3^3\Phi_4\Phi_6^3\Phi_{10}\Phi_{12}\Phi_{15}\Phi_{30}(x_1)\\ 
\hline
\end{array}
\]

\[
\begin{array}{|c|l|}
\hline
|G_{28}|_q = |F_4|_q&\Phi_2\Phi_3(x_1)\Phi_2\Phi_3(y_1)\Phi_2^2\Phi_4\Phi_6(x_1y_1)\Phi_2(x_1y_1^2)\Phi_2(x_1^2y_1)\\
\hline
|F_4 : B_3|_q & \Phi_3(y_1)\Phi_2\Phi_4\Phi_6(x_1y_1)\Phi_2(x_1y_1^2)\\
\hline
|F_4 : A_2 \tilde A_1|_q & \Phi_3(y_1)\Phi_2^2\Phi_4\Phi_6(x_1y_1)\Phi_2(x_1y_1^2)\Phi_2(x_1^2y_1)\\
\hline
|F_4 : A_1 \tilde A_2|_q & \Phi_3(x_1)\Phi_2^2\Phi_4\Phi_6(x_1y_1)\Phi_2(x_1y_1^2)\Phi_2(x_1^2y_1)\\
\hline
|F_4 : C_3|_q & \Phi_3(x_1)\Phi_2\Phi_4\Phi_6(x_1y_1)\Phi_2(x_1^2y_1)\\
\hline
\hline
|G_{29}|_q&\Phi_2^4\Phi_3\Phi_4^4\Phi_5\Phi_6\Phi_8\Phi_{10}\Phi_{12}\Phi_{20}(x_1)\\
\hline
|G_{29} : B_3|_q & \Phi_2\Phi_4^3\Phi_5\Phi_8\Phi_{10}\Phi_{12}\Phi_{20}(x_1)\\
\hline
|G_{29} : A_3|_q & \Phi_2^2\Phi_4^3\Phi_5\Phi_6\Phi_8\Phi_{10}\Phi_{12}\Phi_{20}(x_1)\\
\hline
|G_{29} : A_1 A_2|_q & \Phi_2^2\Phi_4^4\Phi_5\Phi_6\Phi_8\Phi_{10}\Phi_{12}\Phi_{20}(x_1)\\
\hline
|G_{29} : A_3|_q & \Phi_2^2\Phi_4^3\Phi_5\Phi_6\Phi_8\Phi_{10}\Phi_{12}\Phi_{20}(x_1)\\
\hline
|G_{29} : G_{4,4,3}|_q & \Phi_2^2\Phi_4^2\Phi_5\Phi_6\Phi_{10}\Phi_{12}\Phi_{20}(x_1)\\
\hline
\hline
|G_{30}|_q = |H_4|_q&\Phi_2^4\Phi_3^2\Phi_4^2\Phi_5^2\Phi_6^2\Phi_{10}^2\Phi_{12}\Phi_{15}\Phi_{20}\Phi_{30}(x_1)\\
\hline
|H_4 : H_3|_q & \Phi_2\Phi_3\Phi_4^2\Phi_5\Phi_6\Phi_{10}\Phi_{12}\Phi_{15}\Phi_{20}\Phi_{30}(x_1) \\
\hline
|H_4 : H_2 A_1|_q & \Phi_2^2\Phi_3^2\Phi_4^2\Phi_5\Phi_6^2\Phi_{10}^2\Phi_{12}\Phi_{15}\Phi_{20}\Phi_{30}(x_1) \\
\hline
|H_4 : A_2 A_1|_q & \Phi_2^2\Phi_3\Phi_4^2\Phi_5^2\Phi_6^2\Phi_{10}^2\Phi_{12}\Phi_{15}\Phi_{20}\Phi_{30}(x_1) \\
\hline
|H_4 : A_3|_q & \Phi_2^2\Phi_3\Phi_4\Phi_5^2\Phi_6^2\Phi_{10}^2\Phi_{12}\Phi_{15}\Phi_{20}\Phi_{30}(x_1) \\
\hline
\hline
|G_{31}|_q&\Phi_2^8\Phi_3^2\Phi_4^2\Phi_5\Phi_6^4\Phi_{10}^2\Phi_{12}\Phi_{14}\Phi_{18}\Phi_{30}(x_1)\\
\hline
|G_{31} : A_2 A_1|_q & \Phi_2^6\Phi_3\Phi_4^2\Phi_5\Phi_6^4\Phi_{10}^2\Phi_{12}\Phi_{14}\Phi_{18}\Phi_{30}(x_1)
\\
\hline
|G_{31} : A_3|_q & \Phi_2^6\Phi_3\Phi_4\Phi_5\Phi_6^4\Phi_{10}^2\Phi_{12}\Phi_{14}\Phi_{18}\Phi_{30}(x_1)
\\
\hline
|G_{31} : G_{4,2,3}|_q & \Phi_2^3\Phi_3\Phi_4\Phi_5\Phi_6^3\Phi_{10}\Phi_{12}\Phi_{14}\Phi_{18}\Phi_{30}(x_1)
\\
\hline
\hline
|G_{32}|_q&\Phi_2{\Phi''_3}^2{\Phi''_6}{\Phi''''_{12}}{\Phi''''_{30}}(x_1)\Phi_2{\Phi'_3}^2{\Phi'_6}{\Phi'''_{12}}{\Phi'''_{30}}(x_2)\Phi_2^2\Phi_3\Phi_4\Phi_5\Phi_6(x_1x_2)\Phi_2{\Phi'_3}{\Phi'_6}(x_1x_2^2){\Phi'_6}(x_1x_2^3)\Phi_2{\Phi''_3}{\Phi''_6}(x_1^2x_2){\Phi'_3}(x_1^2x_2^3){\Phi''_6}(x_1^3x_2){\Phi''_3}(x_1^3x_2^2)\\
\hline
|G_{32} : G_{25}|_q
& {\Phi''''_{30}}(x_1){\Phi'''_{30}}(x_2)\Phi_2\Phi_4\Phi_5\Phi_6(x_1x_2)\Phi_2{\Phi'_6}(x_1x_2^2){\Phi'_6}(x_1x_2^3)\Phi_2{\Phi''_6}(x_1^2x_2)
{\Phi'_3}(x_1^2x_2^3){\Phi''_6}(x_1^3x_2){\Phi''_3}(x_1^3x_2^2)
\\
\hline
|G_{32} : Z_3G_4|_q
& {\Phi''''_{12}}{\Phi''''_{30}}(x_1){\Phi'''_{12}}{\Phi'''_{30}}(x_2)\Phi_2\Phi_3\Phi_4\Phi_5\Phi_6(x_1x_2)\Phi_2{\Phi'_3}{\Phi'_6}(x_1x_2^2){\Phi'_6}(x_1x_2^3) \Phi_2{\Phi''_3}{\Phi''_6}(x_1^2x_2){\Phi'_3}(x_1^2x_2^3){\Phi''_6}(x_1^3x_2){\Phi''_3}(x_1^3x_2^2)
\\
\hline
\end{array}
\]

\[
\begin{array}{|c|l|}
\hline
|G_{33}|_q&\Phi_2^5\Phi_3^3\Phi_4^2\Phi_5\Phi_6^3\Phi_9\Phi_{10}\Phi_{12}\Phi_{18}(x_1)\\
\hline
|G_{33} : A_3 A_1|_q & \Phi_2^2\Phi_3^2\Phi_4\Phi_5\Phi_6^3\Phi_9\Phi_{10}\Phi_{12}\Phi_{18}(x_1)\\
\hline
|G_{33} : D_4|_q & \Phi_2\Phi_3^2\Phi_5\Phi_6^2\Phi_9\Phi_{10}\Phi_{12}\Phi_{18}(x_1)\\
\hline
|G_{33} : G_3,3,4|_q & \Phi_2^3\Phi_4\Phi_5\Phi_6^2\Phi_{10}\Phi_{12}\Phi_{18}(x_1)\\
\hline
|G_{33} : A_4|_q & \Phi_2^3\Phi_3^2\Phi_4\Phi_6^3\Phi_9\Phi_{10}\Phi_{12}\Phi_{18}(x_1)\\
\hline
\hline
|G_{34}|_q&\Phi_2^6\Phi_3^6\Phi_4^2\Phi_5\Phi_6^6\Phi_7\Phi_8\Phi_9\Phi_{10}\Phi_{12}^2\Phi_{14}\Phi_{15}\Phi_{18}\Phi_{21}\Phi_{24}\Phi_{30}\Phi_{42}(x_1)\\
\hline
|G_{34} : G_{33}|_q & \Phi_2\Phi_3^3\Phi_6^3\Phi_7\Phi_8\Phi_{12}\Phi_{14}\Phi_{15}\Phi_{21}\Phi_{24}\Phi_{30}\Phi_{42}(x_1)\\
\hline
|G_{34} : A_1 A_4|_q
& \Phi_2^3\Phi_3^5\Phi_4\Phi_6^6\Phi_7\Phi_8\Phi_9\Phi_{10}\Phi_{12}^2\Phi_{14}\Phi_{15}\Phi_{18}\Phi_{21}\Phi_{24}\Phi_{30}\Phi_{42}(x_1)\\
\hline
|G_{34} : D_5|_q
& \Phi_2^2\Phi_3^5\Phi_6^5\Phi_7\Phi_9\Phi_{10}\Phi_{12}^2\Phi_{14}\Phi_{15}\Phi_{18}\Phi_{21}\Phi_{24}\Phi_{30}\Phi_{42}(x_1)
\\
\hline
|G_{34} : G_{3,3,4} A_1|_q
& \Phi_2^3\Phi_3^3\Phi_4\Phi_5\Phi_6^5\Phi_7\Phi_8\Phi_{10}\Phi_{12}^2\Phi_{14}\Phi_{15}\Phi_{18}\Phi_{21}\Phi_{24}\Phi_{30}\Phi_{42}(x_1)\\
\hline
|G_{34} : A_3 A_2|_q
& \Phi_2^3\Phi_3^4\Phi_4\Phi_5\Phi_6^6\Phi_7\Phi_8\Phi_9\Phi_{10}\Phi_{12}^2\Phi_{14}\Phi_{15}\Phi_{18}\Phi_{21}\Phi_{24}\Phi_{30}\Phi_{42}(x_1)\\
\hline
|G_{34} : A_5',  A_5''|_q
& \Phi_2^3\Phi_3^4\Phi_4\Phi_6^5\Phi_7\Phi_8\Phi_9\Phi_{10}\Phi_{12}^2\Phi_{14}\Phi_{15}\Phi_{18}\Phi_{21}\Phi_{24}\Phi_{30}\Phi_{42}(x_1)
\\
\hline
|G_{34} : G_{3,3,5}|_q
& \Phi_2^4\Phi_3^2\Phi_4\Phi_6^4\Phi_7\Phi_8\Phi_{10}\Phi_{12}\Phi_{14}\Phi_{15}\Phi_{18}\Phi_{21}\Phi_{24}\Phi_{30}\Phi_{42}(x_1)\\
\hline
\hline
|G_{35}|_q = |E_6|_q & \Phi_2^4\Phi_3^3\Phi_4^2\Phi_5\Phi_6^2\Phi_8\Phi_9\Phi_{12}(x_1)\\
\hline
|E_6 : D_5|_q &\Phi_3^2\Phi_6\Phi_9\Phi_{12}(x_1)\\
\hline
|E_6 : A_4 A_1|_q &\Phi_2\Phi_3^2\Phi_4\Phi_6^2\Phi_8\Phi_9\Phi_{12}(x_1)
\\
\hline
|E_6 : A_2^2 A_1|_q &\Phi_2\Phi_3\Phi_4^2\Phi_5\Phi_6^2\Phi_8\Phi_9\Phi_{12}(x_1)
\\
\hline
|E_6 : A_5|_q &\Phi_2\Phi_3\Phi_4\Phi_6\Phi_8\Phi_9\Phi_{12}(x_1)
\\
\hline
\hline
|G_{36}|_q = |E_7|_q &\Phi_2^7\Phi_3^3\Phi_4^2\Phi_5\Phi_6^3\Phi_7\Phi_8\Phi_9\Phi_{10}\Phi_{12}\Phi_{14}\Phi_{18}(x_1)\\
\hline
|E_7 : E_6|_q & \Phi_2^3\Phi_6\Phi_7\Phi_{10}\Phi_{14}\Phi_{18}(x_1)\\
\hline
|E_7 : D_5 A_1|_q & \Phi_2^2\Phi_3^2\Phi_6^2\Phi_7\Phi_9\Phi_{10}\Phi_{12}\Phi_{14}\Phi_{18}(x_1)\\
\hline
|E_7 : A_4 A_2|_q & \Phi_2^4\Phi_3\Phi_4\Phi_6^3\Phi_7\Phi_8\Phi_9\Phi_{10}\Phi_{12}\Phi_{14}\Phi_{18}(x_1)\\
\hline
|E_7 : A_3 A_2 A_1|_q & \Phi_2^3\Phi_3\Phi_4\Phi_5\Phi_6^3\Phi_7\Phi_8\Phi_9\Phi_{10}\Phi_{12}\Phi_{14}\Phi_{18}(x_1)\\
\hline
|E_7 : A_5 A_1|_q&\Phi_2^3\Phi_3\Phi_4\Phi_6^2\Phi_7\Phi_8\Phi_9\Phi_{10}\Phi_{12}\Phi_{14}\Phi_{18}(x_1)
\\
\hline
|E_7 : A_6|_q & \Phi_2^4\Phi_3\Phi_4\Phi_6^2\Phi_8\Phi_9\Phi_{10}\Phi_{12}\Phi_{14}\Phi_{18}(x_1)\\
\hline
|E_7 : D_6|_q &\Phi_2\Phi_3\Phi_6\Phi_7\Phi_9\Phi_{12}\Phi_{14}\Phi_{18}(x_1)
\\
\hline
\hline|G_{37}|_q = |E_8|_q &\Phi_2^8\Phi_3^4\Phi_4^4\Phi_5^2\Phi_6^4\Phi_7\Phi_8^2\Phi_9\Phi_{10}^2\Phi_{12}^2\Phi_{14}\Phi_{15}\Phi_{18}\Phi_{20}\Phi_{24}\Phi_{30}(x_1)\\
\hline
|E_8 : E_7|_q & \Phi_2\Phi_3\Phi_4^2\Phi_5\Phi_6\Phi_8\Phi_{10}\Phi_{12}\Phi_{15}\Phi_{20}\Phi_{24}\Phi_{30}(x_1)\\
\hline
|E_8 : E_6 A_1|_q&\Phi_2^3\Phi_3\Phi_4^2\Phi_5\Phi_6^2\Phi_7\Phi_8
\Phi_{10}^2\Phi_{12}\Phi_{14}\Phi_{15}\Phi_{18}\Phi_{20}\Phi_{24}\Phi_{30}(x_1)
\\
\hline
|E_8 : D_5 A_2|_q &
\Phi_2^3\Phi_3^2\Phi_4^2\Phi_5\Phi_6^3\Phi_7\Phi_8\Phi_9\Phi_{10}^2\Phi_{12}^2\Phi_{14}\Phi_{15}\Phi_{18}\Phi_{20}\Phi_{24}\Phi_{30}(x_1)\\
\hline
|E_8 : A_4 A_3|_q & \Phi_2^4\Phi_3^2\Phi_4^2\Phi_5\Phi_6^4\Phi_7\Phi_8^2\Phi_9\Phi_{10}^2
\Phi_{12}^2\Phi_{14}\Phi_{15}\Phi_{18}\Phi_{20}\Phi_{24}\Phi_{30}(x_1)\\
\hline
|E_8 : A_4 A_2 A_1|_q & \Phi_2^4\Phi_3^2\Phi_4^3\Phi_5\Phi_6^4\Phi_7\Phi_8^2\Phi_9\Phi_{10}^2
\Phi_{12}^2\Phi_{14}\Phi_{15}\Phi_{18}\Phi_{20}\Phi_{24}\Phi_{30}(x_1)\\
\hline
|E_8 : A_6 A_1|_q & \Phi_2^4\Phi_3^2\Phi_4^3\Phi_5\Phi_6^3\Phi_8^2\Phi_9\Phi_{10}^2
\Phi_{12}^2\Phi_{14}\Phi_{15}\Phi_{18}\Phi_{20}\Phi_{24}\Phi_{30}(x_1)\\
\hline
|E_8 : A_7|_q & \Phi_2^4\Phi_3^2\Phi_4^2\Phi_5\Phi_6^3\Phi_8\Phi_9\Phi_{10}^2
\Phi_{12}^2\Phi_{14}\Phi_{15}\Phi_{18}\Phi_{20}\Phi_{24}\Phi_{30}(x_1)\\
\hline
|E_8 : D_7|_q & \Phi_2^2\Phi_3^2\Phi_4\Phi_5\Phi_6^2\Phi_8\Phi_9\Phi_{10}\Phi_{12}\Phi_{14}\Phi_{15}
\Phi_{18}\Phi_{20}\Phi_{24}\Phi_{30}(x_1)\\
\hline

\end{array}
\]

\end{landscape}

\def\cprime{$'$} \def\cprime{$'$}

\end{document}